\documentclass[a4paper, 11pt]{amsart}
\usepackage{fourier}	
\usepackage{microtype}

\usepackage{verbatim}
\usepackage[english]{babel}
\usepackage[utf8]{inputenc} 		
\usepackage{amssymb}
\usepackage{quiver}		
\usepackage{tikz-cd}				
\usepackage[all]{xy}		
\usepackage{xcolor}
\usepackage[colorlinks=true]{hyperref}  
\usepackage{anysize}						
\marginsize{2.5cm}{2.5cm}{2cm}{2cm}
\usepackage[nameinlink,capitalise,noabbrev]{cleveref}

\newtheorem{theorem}{Theorem}[section]

\newtheorem{lemma}[theorem]{Lemma}
\newtheorem{proposition}[theorem]{Proposition}
\newtheorem{corollary}[theorem]{Corollary}

\newtheoremstyle{TheoremNum}
{}{}              %
{\itshape}                      %
{}                              %
{\bfseries}                     %
{.}                             %
{ }                             %
{\thmname{#1}\thmnote{ \bfseries #3}}%
\theoremstyle{TheoremNum}
\newtheorem{thmn}{Theorem}

\theoremstyle{definition} %

\newtheorem{definition}[theorem]{Definition}
\newtheorem{remark}[theorem]{Remark}

\DeclareMathOperator{\Fun}{Fun}

\DeclareMathOperator{\Hom}{Hom}

\DeclareMathOperator{\id}{id}

\DeclareMathOperator{\Map}{Map}

\newcommand{\APL}{A_{\mathrm{PL}}}
\newcommand{\Ho}{\mathrm{Ho}}

\newcommand{\N}{{\mathbb{N}}}

\newcommand{\Z}{{\mathbb{Z}}}
\newcommand{\Q}{{\mathbb{Q}}}

\newcommand{\C}{{\mathcal{C}}}

\newcommand{\D}{{\mathcal{D}}}
\newcommand{\DD}{{\mathbb{D}}}

\renewcommand{\L}{\mathcal{L}}
\newcommand{\LL}{\mathbb{L}}
\newcommand{\ft}{\textup{ft}}
\newcommand{\op}{\textup{op}}

\newcommand{\sSet}{\mathsf{sSet}}

\newcommand{\DGL}{\mathsf{DGL}}

\newcommand{\Vect}{\mathsf{Vect}}

\newcommand{\CDGA}{\mathsf{CDGA}}
\newcommand{\CDGC}{\mathsf{CDGC}}

\newcommand{\Ch}{\mathrm{Ch}}
\newcommand{\Chz}{\mathrm{Ch}^{\geq 0}}

\newcommand{\Chaz}{\mathrm{Ch}_{\geq 0}}

\newcommand{\V}{{\mathcal{V}}}
\newcommand{\A}{\mathcal{A}}
\newcommand{\B}{\mathcal{B}}
\renewcommand{\S}{\mathcal{S}}
\newcommand{\I}{\mathcal{I}}

\newcommand{\RL}{\mathbb{R}L}
\newcommand{\LR}{\mathbb{L}R}

\newcommand{\fib}{\mathrm{fib}}
\newcommand{\cof}{\mathrm{cof}}

\usepackage{stmaryrd} %

\definecolor{dark-red}{rgb}{0.6,.15,.15}
\definecolor{dark-blue}{rgb}{.1,.1,.65}
\definecolor{dark-green}{rgb}{.1,.65,.1}
\hypersetup{pdfauthor={Bruno Stonek},colorlinks=true, citecolor=dark-blue, urlcolor=dark-green, linkcolor=dark-red,breaklinks=true,hypertexnames=false}

\author{Bruno Stonek}
\address{Faculty of Mathematics, Informatics and Mechanics, University of Warsaw, ul. Banacha 2, 02-097 Warszawa, Poland}
\email{bstonek@mimuw.edu.pl}

\begin{document} 

\title{A homotopical enhancement of Neisendorfer's algebraic models}
\date{}

\begin{abstract} We give an $\infty$-categorical enhancement of Neisendorfer's equivalences between the homotopy categories of rational simplicial sets, commutative dg algebras (cdgas), cocommutative dg coalgebras (cdgcs) and dg Lie algebras (dgls), where these objects have  respective hypotheses of connectedness, nilpotency and finite type. We do so by proving that linear duality is a partial Quillen equivalence between weak finite type connected cdgcs and cdgas, and that Quillen's Chevalley--Eilenberg adjunction is a \emph{partial crossed derivable equivalence} between cdgcs and dgls, a concept we introduce in this paper to axiomatize precisely the behavior of this adjunction.
\end{abstract}

\maketitle

\section{Introduction}

By the mid 1970s, there were concurrent complete algebraic models for the rational homotopy theory of (pointed) spaces: for instance, Quillen's \cite{Quillen-rht} differential graded Lie algebras (\emph{dgl}s) and cocommutative differential graded coalgebras (\emph{cdgc}s), and Sullivan's commutative differential graded algebras (\emph{cdga}s) \cite{FriedlanderGriffithsMorgan1972,Sullivan, Bousfield--Gugenheim}. For Quillen's models, spaces had to be simply-connected; for Sullivan's, connected and nilpotent was enough, %
but on the other hand they had to be finite-type, i.e. their rational homology vector spaces had to be finite-dimensional.

Technically, what Quillen established was a zig-zag of Quillen equivalences of model categories. Bousfield and Gugenheim \cite{Bousfield--Gugenheim} put Sullivan's models in a similar perspective: they established a Quillen adjunction between the model categories of augmented cdgas and (pointed) simplicial sets, %
and proved it is a \emph{partial Quillen equivalence} on the subcategories of cohomologically connected cdgas weakly of finite type (i.e. $H^0=\Q$ and the minimal Sullivan model is of finite type) and connected nilpotent finite-type rational simplicial sets. Partial Quillen equivalences were introduced in \cite{Pepe-Bruno} and succinctly summarize the situation.

Neisendorfer obtained in \cite{Neisendorfer} dgl and cdgc models for connected, nilpotent spaces of finite type, passing via Sullivan's cdgas. (Neisendorfer's and Quillen's dgl and cdgc models agree up to quasi-isomorphism on simply-connected finite-type spaces: %
this follows from Baues--Lemaire's conjecture \cite{BauesLemaire}, proved by \cite{Majewski}; see also \cite{Felix-Fuentes-Murillo-realizations} for a shorter proof.) 
What Neisendorfer proved %
is that the homotopy categories on the following objects are equivalent, %
where in the cdga case the opposite category is understood:

\begin{enumerate}
\item  The connected nilpotent rational pointed simplicial sets of finite type,
    \item The cohomologically connected augmented cdgas  weakly of finite type,
    \item The connected cdgcs  weakly of finite type,
    \item The (non-negatively graded) dgls $L$ such that $H_*(L)$ is nilpotent and of finite type.
    \end{enumerate}

This result appears weaker than Quillen's and Bousfield--Gugenheim's which took the form of a zig-zag of Quillen equivalences or a partial Quillen equivalence, respectively. Indeed,  the homotopy categories of two model categories may be equivalent while the model categories are not Quillen equivalent, see for example \cite{Dugger-Shipley}. Allowing for some more flexibility, we can also word the problem thusly: the underlying $\infty$-categories of model categories with equivalent homotopy categories may not be equivalent. %

Thus, we ask the following question: \textbf{can Neisendorfer’s equivalences be lifted to equivalences of $\infty$-categories?}  We will answer this question in the positive.

An $\infty$-categorical equivalence $F:\C\xrightarrow{\sim} \D$ has concrete advantages over a mere equivalence of the homotopy categories $\Ho(\C)\simeq \Ho(\D)$. To mention just three of them: we get weak equivalences of whole mapping \emph{spaces}, $\Map_\D(FX,FY)\simeq \Map_\C(X,Y)$ \cite[Rmk. 1.2.10.2]{LurieHTT}; we get an equivalence of the homotopy theories of diagrams $F_*: \Fun(I, \C)\xrightarrow{\sim} \Fun(I,\D)$ for a small $\infty$-category $I$ \cite[Prop. 1.2.7.3(2)]{LurieHTT}; %
we get that $F$ preserves (homotopy) limits and colimits \cite[Prop. 5.2.3.5, Prop. 4.2.4.1]{LurieHTT}.

Ours is not the first instance of an enhancement of an already established equivalence of homotopy categories to an equivalence of homotopy theories. For example, Robinson \cite{Robinson} produced an equivalence between the derived category of a ring and the homotopy category of modules over its Eilenberg--Mac Lane spectrum. This was enhanced to a zig-zag of Quillen equivalences in \cite{Schwede-Shipley}. Another example: Beilinson established an equivalence between the bounded derived category of rational polarizable mixed Hodge structures and the corresponding derived category of mixed Hodge complexes \cite{Beilinson}. Drew \cite{Drew} enhanced this to an equivalence of symmetric monoidal $\infty$-categories. He immediately gave a concrete application: constructing cdga models carrying Deligne’s mixed Hodge structures.

We will now describe how to obtain the desired enhancements.

\medskip

The link between cdgas and cdgcs is given by linear duality, denoted $(-)^*$. In order to make it precise, we have to make a bridge between Bousfield--Gugenheim's model category $\CDGA_{/\Q}$ of augmented cdgas and Neisendorfer's model category $\CDGC_1$ of cdgcs $C$ which are connected (i.e. $C_0=\Q$). We thus make use of the category $\CDGA^1$ of cdgas $A$ which are connected (i.e. $A^0=\Q$), with the model structure from \cite[Rmk. 6.7]{Pepe-Bruno}. Its inclusion into $\CDGA_{/\Q}$ is a partial Quillen equivalence (\cref{lem-connected-augmented}). The next step is to connect $\CDGA^1$ with $\CDGC_1$:

\begin{thmn}[\ref{thm-dual-partial}] The adjunction $(-)^*:\CDGC_1\rightleftarrows (\CDGA^1)^\op:(-)^\circ$ is a partial Quillen equivalence on the full subcategories of the weak finite type connected cdgcs or cdgas. 
\end{thmn}

Above, the functor $(-)^\circ$ is a generalization to the dg setting of the classical construction of the Sweedler dual of a (commutative) algebra \cite[Sect. 6]{Sweedler}. We construct it in \cref{prop-adj-*-circ}. In particular, the above adjunction induces a contravariant equivalence of $\infty$-categories between connected weak finite type cdgcs and cdgas; this follows from \cref{prop-infty}, where we prove that partial Quillen equivalences induce equivalences of underlying $\infty$-categories. We note that Roca i Lucio obtained in \cite[Thm. 2.26]{Roca}  an analogous equivalence in an operadic situation which is more general but does not encompass the above, as the commutative dg operad is not cofibrant. The proofs are also different: our proof uses duality of finite-type minimal models. \medskip

Next, we make the link between cdgcs and dgls. Denote by $\DGL$ the category of (non-negatively graded) 
dgls. There is an adjunction $\L:\CDGC_1\rightleftarrows \DGL:\C$ \cite[Sect. B.6]{Quillen-rht}, see also \cite[Sect. 4]{Neisendorfer}, where $\C$ is the dg version of the Chevalley--Eilenberg chain complex for computing Lie algebra homology, and $\L$ outputs the dg Lie algebra of primitive elements in the cobar construction.

Neisendorfer \cite[Prop. 7.2, 7.3]{Neisendorfer} %
proves this adjunction induces an equivalence in homotopy categories, once we restrict to the objects in (3) and (4) above. This equivalence cannot be enhanced to a partial Quillen equivalence $\L\dashv \C$: in fact, $\L\dashv \C$ is not even a Quillen adjunction, as $\L$ does not preserve weak equivalences between cofibrant objects (\cref{rmk-LC-not-quillen}). It does preserve weak equivalences between \emph{fibrant} objects, and $\C$ preserves weak equivalences between \emph{cofibrant} objects. Moreover, $\L$ takes fibrant cdgcs in (3) to cofibrant dgls in (4), and $\C$ takes cofibrant dgls in (4) to fibrant cdgcs in (3). Thus, the derived adjunction $\mathbb{R}\L\dashv \mathbb{L} \C$ on the subcategories (3) and (4) exists, and it satisfies $\mathbb{R}\L(X)=\L(X^\fib)$ and $\mathbb{L} \C(Y)=\C(Y^\cof)$. That is, we have a \emph{right} derived functor of the \emph{left} adjoint and vice versa, and they are mutually inverse equivalences of categories. We encode this in the notion of a \emph{partial crossed derivable equivalence}: see \cref{sect-crossed} for details. 

\begin{thmn}[\ref{prop-LC-partial-mirror}] The adjunction  $\L:\CDGC_1\rightleftarrows \DGL:\C$ is a partial crossed derivable equivalence on the connected cdgcs weakly of finite type and the dgls $L$ such that $H_*(L)$ is nilpotent and of finite type.
\end{thmn}

While much of the algebraic input necessary to establish the above theorem is already present in \cite{Neisendorfer}, it is worthwhile to isolate in what way exactly is the adjunction $\L \dashv \C$ ``homotopical'': to further understanding, but also to deduce that it induces an equivalence of underlying $\infty$-categories. This we prove in \cref{prop-infty} for a general partial crossed derivable equivalence.

Combining the above results, we obtain the main theorem, which is the desired enhancement of Neisendorfer's equivalences of homotopy categories.

\begin{thmn}[\ref{thm-main}] There is a zig-zag of adjunctions
\[\begin{tikzcd}
	\DGL & {\CDGC_1} & {(\CDGA^1)^\op} & {\CDGA_{/\Q}^\op} & {\sSet_*}
	\arrow[""{name=0, anchor=center, inner sep=0}, "\C"', shift right=2, from=1-1, to=1-2]
	\arrow[""{name=1, anchor=center, inner sep=0}, "\L"', shift right=2, from=1-2, to=1-1]
	\arrow[""{name=2, anchor=center, inner sep=0}, "{(-)^*}", shift left=2, from=1-2, to=1-3]
	\arrow[""{name=3, anchor=center, inner sep=0}, "{(-)^\circ}", shift left=2, from=1-3, to=1-2]
	\arrow[""{name=4, anchor=center, inner sep=0}, "{i^\op}", shift left=2, from=1-3, to=1-4]
	\arrow[""{name=5, anchor=center, inner sep=0}, "{L^\op}", shift left=2, from=1-4, to=1-3]
	\arrow[""{name=6, anchor=center, inner sep=0}, "{\langle - \rangle}"', shift right=2, from=1-4, to=1-5]
	\arrow[""{name=7, anchor=center, inner sep=0}, "\APL"', shift right=2, from=1-5, to=1-4]
	\arrow["\dashv"{anchor=center, rotate=-90}, draw=none, from=2, to=3]
	\arrow["\dashv"{anchor=center, rotate=-90}, draw=none, from=1, to=0]
	\arrow["\dashv"{anchor=center, rotate=-90}, draw=none, from=4, to=5]
	\arrow["\dashv"{anchor=center, rotate=-90}, draw=none, from=7, to=6]
\end{tikzcd}\]
where $\L\dashv \C$ is a partial crossed derivable equivalence and the other three adjunctions are partial Quillen equivalences. 
In particular, the $\infty$-categories determined by the full subcategories on the objects (1), (2), (3) and (4) in the list above are all equivalent.
\end{thmn}

The organization of the paper is as follows. In Section 2, we establish some properties of minimal models of cdgas and cdgcs and of the corresponding model categories. In Section 3, we explore linear duality between these categories and prove it defines a partial Quillen equivalence. In Section 4, we define crossed derivable adjunctions and establish their fundamental properties. In Section 5, we explain how partial crossed derivable equivalences and partial Quillen equivalences define equivalences of underlying $\infty$-categories. In Section 6, we turn to dgls, establishing the partial crossed derivable equivalence between them and connected cdgcs. In Section 7, we put these  results together to get the main enhancement theorem.

\subsection{Notations and conventions}
All vector spaces are taken to be over $\Q$. All graded vector spaces are non-negatively graded. A graded vector space or (co)chain complex is said to be of \emph{finite type} if it is finite-dimensional in each degree.  If $V$ is a non-negatively graded vector space, then $sV$ is its suspension %
and if $V$ is positively graded, then $s^{-1}V$ is its desuspension. %

Denote by  $\Chz$ (resp. $\Ch^{\geq1}$) the category of non-negatively (resp. positively) graded cochain complexes over $\Q$ and by $\Chaz$ (resp. $\Ch_{\geq 1}$) the category of non-negatively (resp. positively) graded chain complexes over $\Q$. We consider $\Chz$ and $\Chaz$ as endowed with the standard symmetric monoidal structure whose symmetry has a Koszul sign.

We endow $\Chz$ with the model structure whose weak equivalences are the quasi-isomorphisms, the fibrations are the degreewise surjections and the cofibrations are the degreewise injections in positive degrees. We endow $\Chaz$ with the model structure whose weak equivalences are the quasi-isomorphisms, the fibrations are the degreewise surjections in positive degrees and the cofibrations are the degreewise injections. See \cite[Thm 2.6, 2.7]{Pepe-Bruno}.

If $\C$ is a model category, then we consider $\C^\op$ as a model category with the opposite model structure \cite[Rem. 1.1.7]{Hovey-ModelCategories}. When we write an adjunction as $\C\rightleftarrows \D$, the left adjoint is always on top.

\subsection{Acknowledgments}
The author would like to thank José Manuel Moreno Fernández for early discussions on Neisendorfer's paper.

\section{Minimal algebraic and coalgebraic models}

In this section, we fix notation and recall the necessary results about minimal differential graded (co)commutative (co)algebras. %
We introduce the model structures that we will be using.

\subsection{For cdgas} \label{sect-preliminaries-cdgas}

We first fix some notation: A (unital) commutative algebra object in $\Chz$ is called a \emph{commutative differential graded algebra},
or \emph{cdga} for short. We denote their category by $\CDGA$. 
We denote by $\Q$ the cdga which has $\Q$ concentrated in degree zero.  We will sometimes make use of graded algebras (without a differential), which we will assume to be commutative without explicit mention.

We endow the category $\CDGA$ with Bousfield--Gugenheim's model structure \cite{Bousfield--Gugenheim}. The weak equivalences are the quasi-isomorphisms and the fibrations are the degreewise surjections, so every object is fibrant.

We say a cdga $A$ is  \emph{connected} if the unit $\Q\to A$ induces an isomorphism $\Q\cong A^0$. The full subcategory of $\CDGA$ on the connected cdgas is denoted $\CDGA^1$. We endow it with the model structure from \cite[Rmk. 6.7]{Pepe-Bruno}. The weak equivalences are the quasi-isomorphisms and the fibrations are the degreewise surjections, so every object is fibrant.

We will also consider \emph{augmented} cdgas, i.e. cdgas $A$ endowed with a cdga map $A\to \Q$. We denote their category by $\CDGA_{/\Q}$ and endow it with the overcategory model structure \cite[Page 5]{Hovey-ModelCategories} whereby a map is a (co)fibration or a weak equivalence if and only if it is so in $\CDGA$. Any connected cdga is uniquely augmented, and its differential $d^0$ is the zero map.

In view of the following lemma, we need not specify which of the three categories above we have in mind when we say a connected cdga is cofibrant:

\begin{lemma}\label{lem-cof-cdga}
A connected cdga is cofibrant in $\CDGA^1$ if and only if it is cofibrant in $\CDGA$, if and only if it is cofibrant in $\CDGA_{/\Q}$.
\begin{proof} By construction of the model structure on $\CDGA_{/\Q}$, an augmented cdga is cofibrant in $\CDGA_{/\Q}$ if and only if it is so in $\CDGA$. We will now prove that a connected cdga $A$ is cofibrant in $\CDGA$ if and only if it is cofibrant in $\CDGA^1$.

By \cite[Rmk. 6.8]{Pepe-Bruno} with $G=\{e\}$, there is a Quillen adjunction $L:\CDGA_{/\Q}\rightleftarrows \CDGA^1:i$ where $i$ is the inclusion. Therefore, if $A$ is cofibrant in $\CDGA_{/\Q}$ (equivalently, in $\CDGA$), then $L(A)=L(iA)\cong A$ is cofibrant in $\CDGA^1$, since $L$ is left Quillen. %

Conversely, suppose $A$ is cofibrant in $\CDGA^1$. Let $q:B\to A$ be a cofibrant replacement of $A$ in $\CDGA$. We will produce a connected sub-cdga $B'\subset B$ such that $q|_{B'}:B'\to A$ is still an acyclic fibration in $\CDGA$. This will be enough: indeed, by the lifting axioms in $\CDGA^1$ we can find a lift
\[\begin{tikzcd}
	\Q & {B'} \\
	A & A.
	\arrow[from=1-1, to=1-2]
	\arrow[tail, from=1-1, to=2-1]
	\arrow["{q|_{B'}}","\sim"', two heads, from=1-2, to=2-2]
	\arrow["s", dotted, from=2-1, to=1-2]
	\arrow["{\id_A}"', from=2-1, to=2-2]
\end{tikzcd}\]
Composing $s$ with the inclusion into $B$ proves $A$ is a retract of $B$, which is cofibrant in $\CDGA$, so $A$ is also cofibrant in $\CDGA$.

To produce the desired connected sub-cdga $B'\subset B$, first observe that since $q$ is a quasi-isomorphism, we have $H^0(B)\cong H^0(A)\cong \Q$. So $\ker(d:B^0\to B^1)\cong \Q$. Now choose a vector space $V$ so there is a decomposition $B^1=d(B^0)\oplus V$. Define $(B')^0=\Q$, %
$(B')^1=V$, $(B')^n=B^n$ for $n\geq 2$. This defines a sub-cdga $B'\subset B$. %
The quotient cochain complex $B/B'$ is $(B^0/\Q \xrightarrow{\cong} B^1/V \to 0 \to \cdots)$ where the first map is induced by the differential and is an isomorphism by construction. So $B/B'$ is acyclic, i.e. the inclusion $B'\hookrightarrow B$ is a quasi-isomorphism, so $q|_{B'}:B'\to A$ is a quasi-isomorphism. %
It is also a degreewise surjective map: in degree 0 because of connectivity of $A$ and in degrees $n\geq 2$ because $q^n$ is surjective. As for degree 1: note $q(dB^0)=dq(B^0)=0$, %
so $q^1(V)=q^1(B^1)=A^1$ by surjectivity of $q^1$.  Thus $(q|_{B'})^1$ is also surjective.
\end{proof}
\end{lemma}

A cdga $A$ is \emph{cohomologically connected} if the unit $\Q\to A$ induces an isomorphism $\Q\cong H^0(A)$; any connected cdga is cohomologically connected. We say $A$ is \emph{of finite type} if it is of finite type as a cochain complex, i.e. $A^n$ is a finite-dimensional rational vector space for all $n\geq 0$. 

For a graded vector space $V$, we denote by $\Lambda V$ the free graded commutative algebra on $V$. It has a decomposition by word length, where $\Lambda^rV\subset \Lambda V$ is the subspace spanned by products of $r$ generators. We denote $\Lambda^{\geq r}V:=\bigoplus_{k\geq r} \Lambda^kV$.

A cdga $A$ is \emph{semi-free} if it is isomorphic to one of the form $(\Lambda V,d)$ with $V^0=0$ (in particular, $A$ is connected). A semi-free cdga $(\Lambda V,d)$ is \emph{(Sullivan) minimal} \cite[Sect. 12]{Felix-Halperin-Thomas} if $d$ satisfies both the \emph{nilpotency condition}: $V$ admits an exhaustive filtration by graded subspaces $\{V(k)\}_{k\geq 0}$ such that $dV(0)=0$ and $dV(k)\subset \Lambda V(k-1)$ for $k\geq 1$,  %
and the \emph{minimality condition} $d(V)\subset \Lambda^{\geq 2}(V)$ 
(equivalently, $d(\Lambda^nV)\subset \Lambda^{\geq n+1}V$ for all $n\geq 1$, i.e. the differential increases word length).

Equivalently, one can give a step-by-step ``attaching cells'' definition, as done originally in \cite{DGMS}. %
A connected cdga $A$ is minimal if and only if it admits an exhaustive filtration by sub-cdgas $\{A_r\}_{r\geq -1}$ with $A_{-1}=\Q$ in which every inclusion $A_{r-1}\to A_r$ is the right leg of a pushout in $\CDGA$ (i.e. a relative tensor product) %
\begin{equation}\tag{$\dagger$}\label{eq-cell}\xymatrix{S(V_r) \ar[r]^-{f_r} \ar[d] & A_{r-1} \ar[d] \\ D(V_r) \ar[r] & A_r.}\end{equation}
Here $V_r$ is concentrated in a single degree $\geq 2$ %
and  $f_r(V_r)\subset Z(A_{r-1})\cap \overline{A_{r-1}}^2$, %
where $\overline{A_{r-1}}=A_{r-1}^{\geq 1}$; for a graded vector space $V$ we denote $S(V)=(\Lambda V,0)$ and $D(V)=(\Lambda(s^{-1}V\oplus V),d)$ with $|s^{-1}v|=|v|-1$, $d(s^{-1}v)=v$, $dv=0$. \medskip

Minimal cdgas are cofibrant; %
a \emph{minimal model} of a cdga is a %
weak equivalence to it from a minimal cdga; 
if $A$ is a cohomologically connected cdga, then it admits a minimal model, and these are unique up to isomorphism \cite[Props. 12.1, 12.10]{Felix-Halperin-Thomas}. For other elementary properties of minimal models, see \cite[Ch. 12]{Felix-Halperin-Thomas} or \cite[Sect. 7]{Bousfield--Gugenheim}.

Note that a semi-free cdga $(\Lambda V,d)$ is finite-type if and only if $V$ is of finite type. If $A$ is cohomologically connected, we say it is \emph{weakly of finite type} if its minimal model $(\Lambda V,d)$ is of finite type.%

For an augmented cdga $(A,d)$, we denote by $\overline{A}$ the kernel of the augmentation $A\to\Q$. We denote by $(QA, d_{QA})\in \Chz$ the \emph{indecomposables} of $A$, i.e. the graded vector space $\overline{A}/\overline{A}^2$ together with the induced differential. %
This defines a functor $Q:\CDGA_{/\Q}\to\Ch^{\geq0}$ which is left adjoint to $\Q\oplus Z(-)$, the functor that takes a non-negatively graded cochain complex, endows it with a zero multiplication, and adjoins a  unit. This adjunction is Quillen: indeed, the right adjoint $\Q\oplus Z(-)$ preserves fibrations (degreewise surjections) and weak equivalences. %

\begin{lemma}  \label{lem-QM} Let $A=(\Lambda V,d)$ be a semi-free cdga. Let $d_0$ be the linear part of $d$, %
i.e. the composition $V\xrightarrow{d} \Lambda V \xrightarrow{\pi} V$ where $\pi$ is the projection. Then $(V,d_0)\cong (QA, d_{QA})$. In particular, 
$d_{QA}=0$ if and only if %
$d(V)\subset \Lambda^{\geq 2} (V)$.

\begin{proof} We have $\overline A=\Lambda^{\geq 1}V$ and $\overline{A}^2=\Lambda^{\geq 2} V$. %
Thus, we have a decomposition of graded vector spaces $\overline A=V\oplus \overline{A}^2$, so the projection $V\to QA$, $v\mapsto [v]$ is an isomorphism of graded vector spaces and it commutes with the differentials. 
Finally, $d_{QA}=0$ if and only if $d_0=0$, if and only if $d(V)\subset \overline{A}^2=\Lambda^{\geq 2}V$.%
\end{proof}
\end{lemma}

In particular, the condition $f_r(V_r)\subset Z(A_{r-1})\cap \overline{A_{r-1}}^2$ on the diagram (\ref{eq-cell}) above is equivalent to $Q(f_r)=0$. %

\begin{lemma} \label{lem-Sullivan-cof+d_QA=0}
A cohomologically connected augmented cdga $A$ is %
 minimal if and only if it is cofibrant and the differential of $QA$ is zero.
\begin{proof}The forward direction is  cofibrancy of minimal cdgas combined with the previous lemma. The converse is \cite[Prop. 5.8]{Neisendorfer}. Since Neisendorfer does not prove that result and we could not find a proof in the literature, we produce one. We will prove $A$ is in fact connected, then apply a structure result on cofibrant connected cdgas.

Suppose $A$ is cofibrant and $d_{QA}=0$. Choose a minimal model $m:M\xrightarrow{\sim}A$. Since $Q$ preserves weak equivalences between cofibrant objects, $Qm:QM\to QA$ is a weak equivalence. By \cref{lem-QM}, $d_{QM}=0$. Since moreover $d_{QA}=0$ by assumption, we get that $Qm$ is an isomorphism. %
In particular, $(QA)^0=0$ because  $(QM)^0=0$ by connectedness of $M$.

We now show that $A$ is connected. Let $U$ be the underlying graded vector space of $\overline A$. The surjective trick \cite[Pages 148--149]{Felix-Halperin-Thomas} gives a factorization of $m$ %
\[M\xrightarrow{i}E:=M\otimes D(sU)\xrightarrow{p}A\]
in $\CDGA_{/\Q}$, where $i$ is a weak equivalence and $p$ is
degreewise surjective. %
By 2-out-of-3, the map $p$ is also a weak equivalence. By an analogous lifting argument as in \cref{lem-cof-cdga}, cofibrancy of $A$ gives a section $\sigma:A\to E$ of $p$ in $\CDGA_{/\Q}$. %

Put $J=\ker(A^0\to\Q)$, so $J/J^2=(QA)^0=0$ and hence $J=J^2$. Suppose $\sigma(J)\neq0$. Since $E$ is free as a graded commutative algebra %
and $\sigma$ preserves the augmentation, every element of $\sigma(J)$ has zero constant term. Let $r\geq1$ be the smallest word length occurring %
in an element of $\sigma(J)$, and choose $a\in J$ such that $\sigma(a)$ has a nonzero component of word length $r$. Since $J=J^2$, we can write $a=\sum_{i=1}^N b_ic_i$ with $b_i,c_i\in J$. Thus
\[ \sigma(a)=\sum_{i=1}^N\sigma(b_i)\sigma(c_i).\]
By the choice of $r$, each factor on the right has only terms of word length at least $r$, so the sum has only terms of word length at least $2r$. This contradicts the choice of $a$. Therefore $\sigma(J)=0$. Injectivity of $\sigma$ gives $J=0$, so $A^0=\Q$.

Now that we know that $A$ is connected, we may apply \cite[Prop. 7.11]{Bousfield--Gugenheim}: there is a minimal cdga $N$ and a graded vector space $V$ concentrated in degrees $\geq2$ such that $A\cong N\otimes D(V)$.  %
Passing to indecomposables
gives %
\[
(QA,d_{QA})\cong(QN,0)\oplus(s^{-1}V\oplus V,d),
\qquad d(s^{-1}v)=v,\quad dv=0.
\]
The hypothesis $d_{QA}=0$ forces $V=0$. Thus $A\cong N$, and therefore $A$ is minimal.
\end{proof}
\end{lemma}

\begin{lemma} \label{lem-wft} Let $A\in \CDGA_{/\Q}$ be cohomologically connected. Then the following are equivalent:
\begin{enumerate}
    \item $A$ is weakly of finite type, i.e. its minimal model is of finite type,
    \item $A$ is weakly equivalent to a finite-type cofibrant augmented cdga,
\item $H^*(Q\widetilde A)$ is of finite type, where $\widetilde A\in \CDGA_{/\Q}$ is a cofibrant replacement of $A$.
\end{enumerate}
The analogous connected statement is also true, i.e. we can take $A\in \CDGA^1$ instead of a cohomologically connected $A\in \CDGA_{/\Q}$ and replace ``augmented'' by ``connected'' in (2) and (3). %
\begin{proof}  
If $A$ is weakly of finite type, then in particular it is weakly equivalent to a finite-type cofibrant cdga, so (1) implies (2). We prove (2) implies (1). Suppose the cohomologically connected augmented cdga $B$ is finite-type and cofibrant. We will first prove that the minimal model for $B$ is of finite type.

Let $M=(\Lambda V,d)$ be the minimal model for $B$, so we have a weak equivalence $M\xrightarrow{\sim} B$: since both $M$ and $B$ are cofibrant, $Q$ takes this weak equivalence to a weak equivalence, which takes the form $(V,0)\xrightarrow{\sim} Q(B)$ by \cref{lem-QM}. Therefore, for any $n\geq 1$ we have $V^n\cong H^n(V,0)\cong H^n(Q(B))$ which is finite-dimensional because $B^n$ is finite-dimensional. Thus, the minimal model for $B$ is of finite type.%

Now, if $A$ is weakly equivalent to a finite-type cofibrant augmented cdga $B$, then the minimal models of $A$ and $B$ are isomorphic, so the minimal model of $A$ is finite-type by the above, i.e. $A$ is weakly of finite type. \medskip

We now prove the equivalence of (1) and (3): this result is mentioned at the beginning of p. 52 in \cite{Bousfield--Gugenheim}, but we provide a full proof for completeness. Assume (1) holds. Let $\widetilde{A}\to A$ be a cofibrant replacement of $A$. Let $M=(\Lambda V,d)\to \widetilde A$ be the minimal model of $\widetilde A$, so the composition $M\to A$ also gives the minimal model of $A$. Thus, by assumption, $V$ is of finite type. Since $Q$ %
preserves weak equivalences between cofibrant objects, we have an isomorphism $H^*(QM)\xrightarrow{\cong}H^*(Q\widetilde A)$. But by \cref{lem-QM}, $QM\cong (V,0)$, so $H^*(Q\widetilde A)\cong V$ and thus is of finite type.

Now assume (3) holds: Let $M=(\Lambda V, d)$ be the minimal model of $A$ and let $\widetilde A\to A$ be a cofibrant replacement, so $H^*(Q\widetilde A)$ is of finite-type. By the lifting axioms we obtain a comparison map %
$M\to \widetilde A$ %
which is a weak equivalence by 2-out-of-3; by the same argument as in the previous paragraph, $H^*(QM)\cong H^*(Q\widetilde A)$; also $QM\cong (V,0)$, so $V$ is finite-type, as we wanted. \medskip

As for the connected case: the analogous proof works \emph{mutatis mutandis}. In this case, the indecomposables functor $Q$ takes values in $\Ch^{\geq 1}$. %
\end{proof}
\end{lemma}

\begin{remark} A cohomologically connected augmented cdga may be weakly equivalent to a finite-type cdga (so in particular, it may have finite-type cohomology) but not be weakly of finite type. For example, take $A$ to be the cdga $\Q\oplus \Q x\oplus \Q y$ with $|x|=|y|=1$ and zero differential: %
this is a truncation of the exterior algebra on two elements in degree 1. %
It is the cohomology algebra of the pointed space $S^1\vee S^1$, endowed with the zero differential. Since $S^1\vee S^1$ is formal (as the wedge of two formal spaces, see \cite[Pages 155-156]{Felix-Halperin-Thomas}), %
$A$ is weakly equivalent to the minimal model  $M$ of $S^1\vee S^1$. But $M$ has infinitely many generators of degree 1 %
\cite[Sect. 8.3]{Felix-Halperin-Thomas2}
so it is not finite-type.
\end{remark}

\subsection{For cdgcs}

A (counital) cocommutative coalgebra object in $\Chaz$ is called a \emph{cocommutative differential graded coalgebra}, or \emph{cdgc} for short. We denote their category by $\CDGC$, and denote by $\Q$ the cdgc which has $\Q$ concentrated in degree zero. We will sometimes make use of graded coalgebras (without a differential), which we will assume to be cocommutative without explicit mention.

A cdgc $C$ is \emph{connected} if the counit $C\to \Q$ induces an isomorphism $C_0\to \Q$. Let $\CDGC_1$ denote the full subcategory of $\CDGC$ on the connected cdgcs. Their objects $C$ have canonical coaugmentations, i.e. there is a unique morphism of coalgebras $\Q\to C$, and  morphisms of cdgcs respect them. A cdgc is \emph{of finite type} if it is of finite type as a chain complex.

Let $C\in \CDGC_1$. Denote by $\overline{C}$ the cokernel of the coaugmentation $\Q\to C$. The chain complex of \emph{primitives} of $C$, denoted $(PC, d_{PC})\in \Ch_{\geq 1}$, is defined as the kernel of the induced comultiplication map $\overline{C}\to \overline{C}\otimes \overline{C}$. This is the Eckmann--Hilton dual concept to the indecomposables of a connected cdga.

Neisendorfer \cite[5.2]{Neisendorfer} endows $\CDGC_1$ with a model structure whose weak equivalences are the quasi-isomorphisms and whose cofibrations are the degreewise injections. Every object is cofibrant. 
 We adopt \cite[Def. 5.5]{Neisendorfer}: a connected cdgc $C\in \CDGC_1$ is \emph{minimal} if it is fibrant and $d_{PC}=0$ (dually to the characterization of minimal cdgas in \cref{lem-Sullivan-cof+d_QA=0}). 
Any connected cdgc admits a \emph{minimal model}, i.e. a weak equivalence to a minimal cdgc, 
and these are unique up to isomorphism \cite[Prop. 5.6b]{Neisendorfer}.

For a positively graded vector space $V$, let $\Lambda^c(V)=\bigoplus_{r\geq 0}(V^{\otimes r})^{\Sigma_r}$ be the cofree cocommutative graded coalgebra on $V$, i.e. we make the tensor coalgebra of $V$ cocommutative by considering the fixed points under the actions of the symmetric groups (with Koszul signs). A connected cdgc is \emph{semi-cofree} if it is isomorphic to one of the form $(\Lambda^c V,d)$ with $V^0=0$. Dually to \cref{lem-QM}, the underlying graded vector space to  $P((\Lambda^c(V),d))$ is $V$.

For a graded vector space $V$ concentrated in degrees $\geq 2$, let $S^c(V)=(\Lambda^cV,0)$ and $D^c(V)=(\Lambda^c(s^{-1}V\oplus V),d)$ where %
$d(v)=s^{-1}v$, $d(s^{-1}v)=0$, %
so $P(S^cV)\cong V$, $P(D^cV)\cong s^{-1}V \oplus V$. 
The projection $s^{-1}V\oplus V\to V$ induces a morphism $D^c(V)\to S^c(V)$. If $C\in \CDGC_1$ and $V$ is concentrated in degree $m+1\geq 2$, the left leg of a pullback in $\CDGC_1$ of the form
\[\xymatrix{C'\ar[r] \ar[d] & D^c(V) \ar[d] \\ C \ar[r]^-f & S^c(V)} \]
is called an \emph{elementary extension of $C$ of degree $m$}. %
It is \emph{non-primitive} if $Pf=0$. The vector space $s^{-1}V$ is called the vector space of \emph{generators} of the extension.

A $C\in \CDGC_1$ is \emph{constructible by non-primitive elementary extensions starting from $\Q$} if there is an inverse tower $\Q=D_0 \leftarrow D_1 \leftarrow \cdots$ in $\CDGC_1$ with $C\cong \varprojlim_{m\in \N} D_m$ \footnote{General inverse limits in $\CDGC_1$ are tricky, see \cite[Page 435]{Neisendorfer}.} and for all $m\geq 1$ there are towers $D_{m-1}=E_{m,0}\leftarrow E_{m,1} \leftarrow \cdots$ with $D_m\cong \varprojlim E_{m,j}$ where each $E_{m,j}\to E_{m,j-1}$ is a non-primitive elementary extension of degree $m$.

In words: we can construct $C$ starting from $\Q$, doing possibly infinitely many non-primitive elementary extensions of degree 1, then possibly infinitely many of degree 2, etc. We will now prove that, in the finite-type case, we need only pass to the inverse limit once, at the end. 

\begin{proposition} \label{prop-cdgc-ft-minimal} Let $C\in \CDGC_1$. It is minimal if and only if it is constructible by non-primitive elementary extensions starting from $\Q$. 

Now consider the following condition (M) on $C$: 
There is an inverse tower \[\Q=C_0\leftarrow C_1 \leftarrow C_2 \leftarrow \cdots\] such that:
\begin{enumerate}
    \item  $\varprojlim_{n\in \N} C_n\cong C$,  
    \item Every $c_n:C_{n+1}\to C_n$ is a non-primitive elementary extension by a finite-dimensional vector space of generators concentrated in a single degree $\deg(c_n)\geq 1$, so in particular each $C_n$ is finite-type,
    \item $\deg(c_n)\geq \deg(c_{n-1})$ for all $n\geq 1$ and there are finitely many extensions in the tower of a given degree,
    \item The tower is degree-wise eventually constant: for every $q\geq 0$, the maps $C_{n+1}\to C_n$ are isomorphisms in degrees $\leq q$ for all $n$ sufficiently large.
\end{enumerate}

Condition (M) implies minimality of $C$. Conversely, if $C$ is finite-type, then minimality of $C$ implies condition (M).

\begin{proof}
That minimality is equivalent to constructibility by non-primitive elementary extensions starting from $\Q$ is \cite[Prop. 5.8b]{Neisendorfer}.  Now assume condition (M). Since the degrees of the $c_n$ are non-decreasing, we can group together the extensions of the same dimension: this gives the necessary intermediate towers (adding identities after the last attachment of the given dimension) in the definition of constructibility by non-primitive elementary extensions starting from $\Q$.

Now assume $C\in \CDGC_1$ is finite-type and constructible by non-primitive elementary extensions starting from $\Q$. We will prove the maps in the $\{E_{m,j}\}_j$ towers giving the $D_{m}$ all eventually become isomorphisms, so we can simply concatenate the towers up to that point to get the single tower we want. Moreover, since an extension of dimension $m$ is an isomorphism in degrees less than $m$ and only finitely many non-trivial extensions of dimension at most $q$ occur,  the resulting tower is eventually constant in degrees at most $q$.\medskip

Fix an $m\geq 1$. For every $j\geq 1$, each $E_{m,j}\to E_{m,j-1}$ is the left leg of a pullback as follows, where the vector space of generators $W_{m,j}$ is concentrated in degree $m$:
\[\begin{tikzcd}
	{E_{m,j}} & {D^c(sW_{m,j})} \\
	{E_{m,j-1}} & {S^c(sW_{m,j}).}
	\arrow[from=1-1, to=1-2]
	\arrow[from=1-1, to=2-1]
	\arrow["\lrcorner"{anchor=center, pos=0.125}, draw=none, from=1-1, to=2-2]
	\arrow[from=1-2, to=2-2]
	\arrow["{f_{m,j}}", from=2-1, to=2-2]
\end{tikzcd}\]
We will prove that $W_{m,j}$ is zero for large enough $j$, so their corresponding maps in the tower for $D_{m}$ are isomorphisms. In order to do so, we prove that $(PC)_m\cong \prod_{j\geq 1} W_{m,j}$. Then, since $C$ is of finite type, so is $PC$, and thus only finitely many of the $W_{m,j}$ can be non-zero and they all have to be finite-dimensional.  

Since $P$ is a right adjoint, it preserves limits: applying $P$ to the pullback above gives $PE_{m,j}\cong PE_{m,j-1}\oplus W_{m,j}$. %
Inductively, doing the same to $PE_{m,j-1}$ and successively, gives \[PE_{m,j} \cong PD_{m-1}\oplus W_{m,1}\oplus \dots \oplus W_{m,j}.\]
Therefore,
\[PD_m\cong \varprojlim_j PE_{m,j} \cong \varprojlim_j (PD_{m-1} \oplus W_{m,1} \oplus \dots \oplus W_{m,j}) \cong PD_{m-1}\oplus \prod_{j\geq 1} W_{m,j}.\]
Since $PD_0=0$ and $W_{m,j}$ is concentrated in degree $m$, this shows $PD_m$ is concentrated in degrees $\leq m$,  $(PD_m)_m\cong \prod_{j\geq1} W_{m,j}$ and the maps $(PD_r)_m\to (PD_{r-1})_m$ are isomorphisms for all $r\geq m+1$. From this, we deduce what we wanted to prove:
\[(PC)_m\cong (\varprojlim_r PD_r)_m \cong \varprojlim_r (PD_r)_m\cong (PD_m)_m \cong \prod_{j\geq 1}W_{m,j}. \qedhere \]
\end{proof}
\end{proposition}

We will make use of an additional property of towers as above. We isolate it as a lemma on general degree-wise eventually constant towers:

\begin{lemma}\label{lem-tower} Let $C_0\leftarrow C_1\leftarrow \cdots$ be an inverse tower of connected cdgcs which is degree-wise eventually constant. Then the inverse limit $C:=\varprojlim_{n\in \N} C_n$ can be computed degreewise, so in particular the projections $C\to C_n$ are isomorphisms in degrees $\leq q$ for all $n$ sufficiently large.

Moreover, if $(D\in\CDGC_1, (\lambda_j:D\to C_j)_{j\geq 0})$ is a cone for this tower and for every $q$ the $\lambda_j$ are isomorphisms in degrees $\leq q$ for all $j$ sufficiently large, then the induced map $D\to\varprojlim_jC_j$ is an isomorphism.
\begin{proof} %
Let $L$ be the inverse limit of the underlying inverse tower of chain complexes, which is computed degreewise. Let $\pi_n:L\to C_n$ denote the projections. Fix $q\geq 0$. Once the tower has stabilized in degrees $\leq q$, each $\pi_n$ is an isomorphism in these degrees: a compatible sequence is uniquely recovered from its component at stage $n$, using transition maps towards earlier stages and their inverses towards later stages.

Observe that $(\pi_n\otimes\pi_n)_q$ is also an isomorphism for sufficiently large $n$. Define $\Delta_L:L\to L\otimes L$ via
\[(\Delta_L)_q
:=
(\pi_n\otimes\pi_n)_q^{-1} \circ (\Delta_{C_n})_q  \circ (\pi_n)_q:L_q\to (L\otimes L)_q.\]
This is independent of $n$ because the transition maps preserve coproducts. Define the counit similarly. %
The coalgebra identities and compatibility with the differential follow, since they are satisfied for the $C_n$.

Finally, a compatible family of cdgc morphisms $X\to C_n$ induces a unique chain map $X\to L$. This map preserves coproducts and counits, as can again be checked after applying projections at a sufficiently late stage. Thus $L$ has the universal property of the inverse limit in $\CDGC_1$. Therefore, $C\cong L$, and the projections $C\to C_n$ are isomorphisms in degrees $\leq q$ for all sufficiently large $n$.

We prove the last statement. Let $\theta:D\to L$ be the induced map. Fix $q$ and choose $j$ sufficiently large so both $
(\lambda_j)_q:D_q\to(C_j)_q,  (\pi_j)_q:L_q\to (C_j)_q$
are isomorphisms. Since $(\pi_j)_q \circ \theta_q=(\lambda_j)_q$, the map $\theta_q$ is an isomorphism. This holds in every degree, so $\theta$ is a cdgc isomorphism.
\end{proof}
\end{lemma}

\begin{remark}
Since a minimal cdgc $C$ is constructible by elementary extensions starting from $\Q$, it is in particular semi-cofree: %
$C\cong (\Lambda^c V,d)$ for $V=\bigoplus_{m\geq 0}(PC)_m$%
. The condition $d_{PC}=0$ translates to: $d$ decreases word length, %
i.e. $d((V^{\otimes r})^{\Sigma_r})\subset \bigoplus_{1\leq j<r} (V^{\otimes j})^{\Sigma_j}$ for $r\geq 1$.%
\end{remark}

\begin{remark}
If the minimal cdgc $C$ is not finite-type, then it may be impossible to present $C$ as the inverse limit of a single $\N$-indexed tower of non-primitive elementary extensions as in \cref{prop-cdgc-ft-minimal}. The asymmetry with the minimal cdga case where one can always take an $\N$-indexed filtration comes from the fact that for a generator $v$ of $(\Lambda(V),d)$, the element $dv$ is a polynomial, so it involves finitely many earlier generators, whereas the coalgebraic dual statement is no longer true. %
More precisely: on a free graded algebra $\Lambda(V)$ a derivation $d$ is specified by a linear map $d|_V:V\to \Lambda(V)$ of degree 1, and  it satisfies the following property: for every finite-dimensional graded subspace $i:E\hookrightarrow V$, there exist a finite-dimensional graded subspace $j:U\hookrightarrow V$ and a linear map $a:E\to\Lambda(U)$ of degree 1 making the following diagram commute:
\[\begin{tikzcd}
	E & V \\
	{\Lambda(U)} & {\Lambda(V).}
	\arrow["i", from=1-1, to=1-2]
	\arrow["a"', from=1-1, to=2-1]
	\arrow["{d|_V}", from=1-2, to=2-2]
	\arrow["{\Lambda(j)}"', from=2-1, to=2-2]
\end{tikzcd}\]
On the other hand, on a cofree graded coalgebra $\Lambda^c(V)$, a coderivation is specified by a linear map $\overline{d}:\Lambda^c(V)\to V$ of degree $-1$ \cite[Lem. 1.9]{Neisendorfer}. But the dual statement to the above, replacing $\Lambda$ by $\Lambda^c$, reversing the arrows and replacing subspaces by quotients, does not hold in general.

We can also give a concrete example. Take $C=(\Lambda^c(W),0)$ with $W=\bigoplus_{i\geq 0} \Q w_i$, $|w_i|=1$ for all $i$. It is not of finite type. Suppose it can be presented by a tower as in \cref{prop-cdgc-ft-minimal}. Call $W_n$ the vector space of generators at stage $n$. Computing primitives as in the proof  gives $PC\cong \prod_{n\geq 0} W_n$. In degree 1, this gives $W\cong \prod_{n\geq 0} (W_n)_1$. This is impossible: since $W$ is infinite-dimensional, infinitely many of the $(W_n)_1$ would have to be non-zero, in which case their product becomes uncountable, unlike $W$ which is countable.
\end{remark}

A connected cdgc is \emph{weakly of finite type} if its minimal model $M$ is finite type. The dual of \cref{lem-wft} is as follows. Its proof dualizes directly.

\begin{lemma} \label{lem-wft-cdgc} Let $C\in \CDGC_1$. Then the following are equivalent:
\begin{enumerate}
    \item $C$ is weakly of finite type, i.e. its minimal model is of finite type,
    \item $C$ is weakly equivalent to a finite-type fibrant connected cdgc,
\item $H_*(P\widetilde C)$ is of finite type, where $\widetilde C\in \CDGC_1$ is a fibrant replacement of $C$.
\end{enumerate}
\end{lemma}

\subsubsection{Some remarks on the model structure}

We can reinterpret Neisendorfer's model structure on $\CDGC_1$ as follows.  
Let $J:\CDGC_1\to \Ch_{\geq1}$ be the coaugmentation coideal functor which takes a connected cdgc $C$ to the cokernel of its coaugmentation. In other words, $J(C)$ is the truncation of $C$ to positive degrees. %
The right adjoint to the functor $J$ is the cofree functor. %
Now endow $\Ch_{\geq 1}$ with a model structure via the shifting isomorphism $\Ch_{\geq 1}\cong  \Chaz$. Neisendorfer's model structure is then the left-transferred model structure via $J$, i.e. a map in $\CDGC_1$ is a weak equivalence or a cofibration if its truncation to positive degrees is so in $\Ch_{\geq 1}$.
Note as well that the primitives functor $P:\CDGC_1\to \Ch_{\geq 1}$ is right Quillen, with left Quillen adjoint given by $V\mapsto \Q\oplus V$ where all of $V$ is primitive.

Neisendorfer's model structure on $\CDGC_1$ shows that connected cdgcs do not exhibit the problem pointed out in \cite[Rmk. 1.41]{Roca}, whereby dg coalgebras over an operad may not admit the left-transferred model structure along the forgetful functor. The author observes (without proof) that the category of unbounded cdgcs furnishes such an example.  We now observe that the category $\CDGC$ of non-negatively graded cdgcs also does not admit the left-transferred model structure along the forgetful functor $\CDGC\to \Chaz$. This is in contrast with $\CDGA$, which does have the right-transferred model structure along the forgetful functor $\CDGA\to \Chz$ \cite{Bousfield--Gugenheim}, i.e. a map in $\CDGA$ is a weak equivalence or a fibration if it is so in $\Chz$.

\begin{proposition}
    The model structure on $\CDGC$ with quasi-isomorphisms as weak equivalences and  degreewise injections as cofibrations does not exist.
    \begin{proof}
        This is essentially \cite[Prop. 8.10]{LeGrignouLejay2022}, but there the cdgcs are unbounded and over an algebraically closed field, so we explain how to use their result to get what we want. In their proof, the authors prove the following statement: let $k$ be an algebraically closed field of characteristic zero. Let $\C$ be the category of unbounded cdgcs over $k$. Then the sum map $k\oplus k\to k$ cannot be factored in $\C$ through a cdgc via a degreewise injective map followed by a quasi-isomorphism.%

        Now, suppose the model structure in $\CDGC$ in the statement existed. Then the sum $\Q\oplus \Q\to \Q$ would factor through a cdgc $C$ via a degreewise injective map followed by a quasi-isomorphism. %
        Take $k$ to be an algebraic closure of $\Q$. The extension of scalars functor $-\otimes k:\CDGC\to \C$ %
        preserves degreewise injectivity of maps since we are over a field. It also preserves quasi-isomorphisms since its underlying functor of chain complexes is an exact functor, so it commutes with homology. Thus, this factorization would yield a factorization in $\C$ as in the previous paragraph, and those do not exist.
    \end{proof}
\end{proposition}

\section{Linear duality of cdgas and cdgcs} \label{sect-duality}

Linear duality of finite-dimensional vector spaces is a contravariant equivalence of categories. Linear duality also gives a contravariant equivalence between finite-dimensional (commutative) algebras and finite-dimensional (cocommutative) coalgebras. We now wish to extend this to an equivalence of the homotopy theories of finite-type connected cdgas and cdgcs. We will achieve this via a partial Quillen equivalence, a notion from \cite{Pepe-Bruno} whose definition we recall below.

In the first subsection, we prove some properties of dg linear duality. In the second subsection, we introduce the dg Sweedler dual, which extends linear duality to cdgas not necessarily of finite type, and we establish the desired partial Quillen equivalence.

\subsection{Linear duals}

If $V$ is a rational vector space, we let $V^*$ denote its linear dual $\Hom_\Q(V,\Q)$. Now let $(V,d)\in \Chaz$. Its \emph{linear dual} is the cochain complex $(V^*, \delta)\in \Chz$ defined by \[(V^*)^n:=V_n^*=\Hom_\Q(V_n,\Q),\] with differential $\delta$ given by $\delta:V_n^*\to V_{n+1}^*$, $\delta(\alpha)=\alpha \circ d$. We refer to the elements of $V^*$ as \emph{functionals}. %
If $f:V\to W$ is a morphism in $\Chaz$, then $f^*:W^*\to V^*$ defined by $f^*(\alpha)=\alpha \circ f$ is a cochain map, and thus we get a linear dual functor \[(-)^*:\Chaz^\op\to \Chz.\]

We will use the following construction below. For a graded vector space $V$, let $\langle -,-\rangle:V^*\otimes V\to \Q$ denote the pairing given by
    \[
    \langle \alpha, v\rangle = \begin{cases}\alpha(v) & \text{if } \alpha\in V^*_p, v\in V_p \text{ for some } p,\\ 0 & \text{else.}\end{cases}
    \] By abuse of notation, if $W$ is another graded vector space, also let $\langle -,-\rangle:(V^*\otimes W^*) \otimes (V\otimes W)\to \Q$ denote the pairing given by the composition
    \[\xymatrix{(V^*\otimes W^*) \otimes (V\otimes W) \ar[rr]^-{\id\otimes \tau \otimes \id} && V^*\otimes V \otimes W^*\otimes W \ar[rr]^-{\langle-,-\rangle \otimes \langle -,-\rangle} && \Q\otimes \Q \ar[r]^-\mu & \Q}\]
    where $\tau$ exchanges the middle factors (without any signs) and $\mu$ is the product. Explicitly, this is the bilinear extension of the following expression on elementary tensors:
    \[\langle \alpha \otimes \beta, v \otimes w\rangle = \begin{cases}\alpha(v)\beta(w) & \text{if } \alpha\in V^*_p, \beta\in W^*_q, v\in V_p, w\in W_q \text{ for some } (p,q),\\ 0 & \text{else.}\end{cases}\]

\begin{proposition} \label{prop-lax}The functor $(-)^*:\Chaz^\op\to \Chz$ is lax symmetric monoidal.
\begin{proof} Let $V,W\in \Chaz$. We define the structure map $\nabla:V^*\otimes W^*\to (V\otimes W)^*$ as follows. For $\alpha\in V_p^*$, $\beta\in W_q^*$, we define the functional $\nabla(\alpha\otimes \beta)\in (V\otimes W)_{p+q}^*$ as follows:
\[\nabla(\alpha\otimes \beta)(v\otimes w)=\langle \alpha\otimes \beta, v\otimes w\rangle.\]%
It is clear that $\nabla$ is a map of graded vector spaces and it is a straightforward computation to see it commutes with the differentials. 

We define the unit structure map to be $\nabla_0:\Q\to \Q^*$, $\nabla_0(a)(b)=ab$. With these definitions, it is clear that the $\nabla$ maps are natural transformations, and that together with $\nabla_0$ they satisfy the associativity and unit coherence axioms. The symmetry axiom takes the form $\nabla=\sigma^*\circ \nabla\circ \sigma$, where $\sigma:V^*\otimes W^*\to W^*\otimes V^*$ is the symmetry $\sigma(\alpha\otimes \beta)=(-1)^{pq}\beta\otimes \alpha$, for $\alpha \in V_p^*$ and $\beta\in W_q^*$. It is quickly verified: indeed $\sigma^*\circ \nabla \circ \sigma$ reduces to $\nabla$ since the factor $(-1)^{pq}$ appears two times, so it multiplies to 1.
\end{proof}
\end{proposition}

\begin{corollary} \label{cor:lineardual}
    There is a linear dual functor $(-)^*:\CDGC\to \CDGA^\op$, and it restricts to a functor $(-)^*:\CDGC_1\to (\CDGA^1)^\op$.
    \begin{proof}
        Lax symmetric monoidal functors induce functors between the corresponding categories of commutative monoids. On the other hand, if $\V$ is a symmetric monoidal category, then there is an isomorphism between the category of commutative monoids in $\V^\op$ and the opposite to the category of cocommutative comonoids in $\V$; in particular, the category of commutative monoids in $\Chaz^\op$ is isomorphic to $\CDGC^\op$. Since the category of commutative monoids in $\Chz$ is $\CDGA$, we get a linear dual functor $(-)^*:\CDGC^\op\to \CDGA$ or, equivalently, a linear dual functor as in the statement.%

        Now, let $C\in \CDGC_1$, so $\epsilon_0:C_0\to \Q$ is an isomorphism. Therefore, $\epsilon_0^*:\Q^*\to C_0^*$ is also an isomorphism, so the component at zero of the unit of $C^*$, namely $\Q\stackrel{\nabla_0}{\to}\Q^* \stackrel{\epsilon_0^*}{\to}C_0^*$, is also an isomorphism. Thus, $C^*\in \CDGA^1$.
    \end{proof}
\end{corollary}

\begin{remark}

    We can make the product and the unit in the linear dual of a cdgc $C$ more explicit by inspection. The product map for $C^*$ is $C^*\otimes C^* \stackrel{\nabla}{\to}(C\otimes C)^* \stackrel{\Delta^*}{\to} C^*$, where $\Delta$ is the coproduct of $C$. More explicitly: if $\alpha \in C_p^*, \beta\in C_q^*, v\in C_{p+q}$, and the $(p,q)$-component of the coproduct of $v$ is $\sum_i x_i \otimes y_i$ with $x_i\in C_p$, $y_i\in C_q$, then
    \[(\alpha \cdot \beta)(v)=\sum_i \alpha(x_i)\beta(y_i).\]
    For a homogeneous $v\notin C_{p+q}$, the above product is zero. We can also describe $\alpha\cdot \beta$ in a more concise way using the pairings introduced above:  \begin{equation}\label{eq-product}\langle \alpha\cdot \beta,v \rangle=\langle \alpha \otimes \beta, \Delta v\rangle\end{equation} for all $v\in C$. %

    The unit map of $C^*$ is the composition $\Q\stackrel{\nabla_0}{\to}\Q^* \stackrel{\epsilon^*}{\to}C^*$, where $\epsilon:C\to \Q$ is the counit of $C$, i.e. the unit element of $C^*$ is the functional $\epsilon$.

\end{remark}

\begin{remark}We have chosen to give a direct proof of \cref{cor:lineardual} since it is very explicit. Alternatively, the result can be proved in a more abstract way as in \cite[Lem. 2.1]{Roca}, which applies to the unbounded situation.
\end{remark}

In the following proposition, we use the subindex ``$\ft$'' to denote full subcategories on finite-type objects. %

\begin{proposition} \label{prop-lineardual-ft}
    The linear dual of a finite-type cdga is a finite-type cdgc. This defines a functor $(-)^*:\CDGA_\ft^\op\to \CDGC_\ft$, and \[(-)^*:\CDGC_\ft\rightleftarrows \CDGA_\ft^\op:(-)^*\] is an adjoint equivalence.
    \begin{proof} We can restrict the linear dual functor $(-)^*:\Chaz^\op\to \Chz$ from \cref{prop-lax} to the subcategory of finite-type chain complexes. It then lands in the subcategory of finite-type cochain complexes, and it becomes an equivalence, with pseudoinverse given by the linear dual functor in the other direction.

    The monoidal structure maps from \cref{prop-lax} become isomorphisms in the subcategory of finite-type objects (the fact that our chain complexes are bounded below becomes crucial here), so the lax symmetric monoidal structure on $(-)^*$ from that proposition becomes strong. There is an analogous strong symmetric monoidal structure on the linear dual functor in the other direction. Moreover, the canonical double-dual evaluation isomorphisms are
monoidal. Thus, the adjoint equivalence $(-)^*:(\Chaz)_\ft^\op\to \Chz_\ft:(-)^*$ of strong symmetric monoidal functors induces an adjoint equivalence $(-)^*:\CDGC_\ft^\op\rightleftarrows\CDGA_\ft:(-)^*$  by passing to the categories of commutative monoids. Taking opposite categories gives the asserted adjoint equivalence.
    \end{proof}
\end{proposition}

\begin{remark}
    Let $A$ be a finite-type cdga. We can make the coproduct %
    of $A^*$ more explicit. The finite-type hypothesis turns $\nabla:A^*\otimes A^*\to (A\otimes A)^*$ into an isomorphism, so we define $\Delta:=\nabla^{-1} \circ \mu^*$  where $\mu^*:A^*\to (A\otimes A)^*$ is the linear dual of the multiplication of $A$.  
    Using pairings, this means $\Delta \alpha$ for $\alpha\in A^*$ is determined by the expression
    \[\langle \Delta \alpha, a\otimes b\rangle = \langle \alpha, ab\rangle\]
    for all $a,b\in A$.  Notice how similar it is to the expression in (\ref{eq-product}).
\end{remark}

In relating minimal models of cdgas and cdgcs, we will need the following simple lemma \cite[1.2.4]{LodayVallette}: %

\begin{lemma} \label{lem-PQ} For a finite-type connected cdga $A$ and a finite-type connected cdgc $C$, we have natural isomorphisms of (co)chain complexes:
\[P(A^*)\cong (QA)^*, \qquad Q(C^*) \cong (PC)^*. \]
\end{lemma}

\begin{proposition} \label{prop-dual-ft-minimal} The linear dual of a finite-type minimal %
cdgc (resp. %
cdga) is a finite-type minimal cdga (resp. cdgc).
\begin{proof}
$\bullet$ Let $C$ be a finite-type minimal %
cdgc, so it has a tower $\Q=C_0\leftarrow C_1 \leftarrow C_2 \leftarrow \cdots$ as in \cref{prop-cdgc-ft-minimal}. 
Fix a degree $q$, so the tower in degree $q$ stabilizes after some $N$ steps. By \cref{lem-tower} the projection $C\to C_N$ induces an isomorphism in degree $q$. Consider the  direct system of vector spaces $\{(C_n)_q^*\}_n$: it consists of isomorphisms from stage $N$ onwards. Therefore,
\[
(\varinjlim_n C_n^*)^q
\cong\varinjlim_n(C_n)_q^*
\cong(C_N)_q^*
\xrightarrow[\cong]{(\pi_N)_q^*}C_q^*.
\]
The first identification uses that filtered colimits of cdgas are computed degreewise. %
The composite is the component in degree $q$ of the canonical morphism $\theta:\varinjlim_n C_n^*\longrightarrow C^*$. 
Thus $\theta$ is degreewise an isomorphism, hence an isomorphism of cdgas.

It now suffices to make the following general observation: for a finite-type cdgc $E$,  %
the linear dual of the pullback in a non-primitive elementary extension of $E$ with finite-dimensional vector space of generators $s^{-1}V$ is a pushout as in (\ref{eq-cell}).

Indeed, first note that the pullback of such an elementary extension takes place in $\CDGC_\ft$, which is closed under pullbacks in $\CDGC$. 
By \cref{prop-lineardual-ft}, linear duality takes it to a pushout in $\CDGA_\ft$, which is closed under pushouts in $\CDGA$.  Now note $S^c(V)^*\cong S(V^*)$ and $D^c(V)^*\cong D(V^*)$.

Thus, the linear dual of the pullback of the elementary extension $C_{n+1}\to C_n$ becomes the pushout in $\CDGA$
\[\begin{tikzcd}
	{S(V_n^*)} & {C_n^*} \\
	{D(V_n^*)} & {C_{n+1}^*}
	\arrow["{f_n^*}", from=1-1, to=1-2]
	\arrow[from=1-1, to=2-1]
	\arrow[from=1-2, to=2-2]
	\arrow[from=2-1, to=2-2]
	\arrow["\lrcorner"{anchor=center, pos=0.125, rotate=180}, draw=none, from=2-2, to=1-1]
\end{tikzcd}\]
where $s^{-1}V_n$ is the vector space of generators of $C_{n+1}\to C_n$. By \cref{lem-PQ}, $Q(f_n^*)=P(f_n)^*$ which is zero because $C_{n+1}\to C_n$ is non-primitive.%

$\bullet$ Let $A$ be a finite-type minimal %
cdga. The proof above dualizes to prove that $A^*$ is also minimal (and finite-type, of course). %
Alternatively, we can give a model-categorical proof using a result we will prove in a second. Since $A$ is minimal, it is cofibrant %
and the differential on $QA$ is zero. By \cref{prop-dual-quillen} its %
linear dual is fibrant and it is also finite-type. By \cref{lem-PQ}, $P(A^*)\cong (QA)^*$. Thus, the differential on $P(A^*)$ is zero, so $A^*$ is minimal, as we wanted to show.
\end{proof}
\end{proposition}

\subsection{Linear duality as a partial Quillen equivalence}

We now aim to realize the duality of finite-type cdgas and cdgcs as a partial Quillen equivalence in the sense of \cite{Pepe-Bruno}. In order to do so, we first need a right adjoint to the linear duality functor on all cdgas, not necessarily of finite type:

\begin{proposition} \label{prop-adj-*-circ}
    The functor $(-)^*:\CDGC\to \CDGA^\op$ admits a right adjoint $(-)^\circ:\CDGA^\op\to \CDGC$, the \emph{Sweedler dual}. This adjunction restricts to an adjunction \[(-)^*:\CDGC_1\rightleftarrows (\CDGA^1)^\op:(-)^\circ.\]
    The Sweedler dual of a finite-type cdga is its linear dual.
    \begin{proof}
        The first part is a particular case of \cite[Prop. 2.2]{Roca} in the non-negatively graded context: there, an abstract proof of existence of this adjoint is produced. We shall provide instead a generalization to the dg setting of the classical construction of the Sweedler dual of an algebra \cite{Sweedler} to obtain an explicit description of this functor directly: %
        the Sweedler dual of a dga $A$ consists, by definition, of the functionals on $A$ such that their kernels contain a dg ideal of finite codimension. %
        By the universal property of the quotient, these are the functionals on $A$ that factor via the quotient epimorphism $\pi:A\to A/I$ for a dg ideal $I$ of finite codimension; in other words, functionals in the image of the monomorphism $\pi^*:(A/I)^*\hookrightarrow A^*$. %
        Identifying each $(A/I)^*$ with its image in $A^*$, we have $A^\circ = \bigcup\limits_{I\in \I(A)}(A/I)^*$ where $\I(A)$ is the set of dg ideals of $A$ of finite codimension. Degreewise,
        \[(A^\circ)_n=\{\alpha\in (A^n)^*:\alpha(I^n)=0 \text{ for some } I\in \I(A)\}.\]%
The coproduct of an $\alpha\in (A^\circ)_n$ is obtained as follows: take an $I\in \I(A)$ such that $\alpha(I^n)=0$, so we get an induced $\widetilde\alpha:A^n/I^n\to \Q$. This defines an element in $(A/I)^*$: since $A/I$ is finite-dimensional, $(A/I)^*$ has a coproduct $\Delta \widetilde\alpha\in (A/I)^*\otimes (A/I)^*\subset A^\circ \otimes A^\circ$ %
(\cref{prop-lineardual-ft}). This is independent of the choice of $I$: if $J$ is another
choice, then $I\cap J$ is again a dg ideal of finite
codimension, and the canonical inclusions $(A/I)^*,(A/J)^*\hookrightarrow(A/(I\cap J))^*$
are dg coalgebra morphisms. Define $\Delta\alpha$ to be the image of $\Delta\widetilde\alpha$ under the inclusion $(A/I)^*\otimes(A/I)^*\hookrightarrow A^\circ\otimes A^\circ$. 
The differential on $A^\circ$ is the restriction of the differential on $A^*$, and the counit  $A^\circ\to \Q$ is evaluation on the unit of $A$. 

This defines a dg coalgebra structure on $A^\circ$ which agrees with that on each $(A/I)^*$. When $A$ is commutative, $A^\circ$ is a cdgc, and this is functorial. We succinctly establish the adjunction. If $f:A\to C^*$ is a cdga morphism, then its transpose $C\to A^\circ$ is defined by $c\mapsto(a\mapsto f(a)(c))$. This does land in $A^\circ$: %
for $c\in C$, choose a finite-dimensional dg subcoalgebra $E\subset C$ containing $c$ \cite[Cor. 1.6]{Getzler-Goerss}. The functional $a\mapsto f(a)(c)$ annihilates the dg ideal
$  \ker(A\stackrel{f}{\to}C^*\to E^*)$ which has finite codimension because $E^*$ is finite-dimensional. %
The transpose of $g:C\to A^\circ$ is the cdga morphism $A\to C^*$, $a\mapsto(c\mapsto g(c)(a))$. These two transpose correspondences are natural and inverse to each other.

To see that this adjunction restricts as stated, it suffices to prove that the Sweedler dual of a connected cdga $A$ is connected, since $\CDGC_1\subset \CDGC$ and $\CDGA^1 \subset \CDGA$ are full subcategories. Since $(A^\circ)_0\subset (A^0)^*\cong \Q$ is a subspace, it suffices to produce a non-zero element in $(A^\circ)_0$. Consider $\epsilon:A\to \Q$, the augmentation of $A$. Its kernel is a dg ideal such that $A/\ker(\epsilon)\cong \Q$, so it has codimension 1. Therefore, $\epsilon^0$ belongs to $(A^\circ)_0$, %
and it is an isomorphism, so it is the desired non-zero element.

For the last statement, observe that if $A$ is a finite-type cdga and $\alpha\in (A^n)^*$, then $I:=\bigoplus_{j>n} A^j$ has finite codimension and is annihilated by $\alpha$, so $A^\circ=A^*$.
    \end{proof}
\end{proposition}

\begin{proposition} \label{prop-dual-quillen} The adjunction $(-)^*:\CDGC_1\rightleftarrows (\CDGA^1)^\op:(-)^\circ$ is Quillen. Moreover, the functor $(-)^*$ preserves and reflects quasi-isomorphisms. %
\end{proposition}

The proof of the first part of this proposition is like that of \cite[Lem. 2.22]{Roca}, which concerns (co)algebras over a cofibrant dg operad given as the operadic cobar construction of a conilpotent curved cooperad (this is not the case in our situation). %

\begin{proof} Let $f:C\to D$ be a cofibration in $\CDGC_1$, i.e. a degreewise injective map. Its dual $f^*:D^*\to C^*$ is a cofibration in $(\CDGA^1)^\op$, i.e. a fibration in $\CDGA^1$, since the dual of an injective map of vector spaces is surjective. %

We now observe that if a morphism in $\CDGC_1$ is a quasi-isomorphism then its linear dual is so, too, so a fortiori the functor $(-)^*$ preserves acyclic cofibrations. Indeed, %
more generally $H^n(C^*)\cong H_n(C)^*$ naturally in $C\in \Chaz$.

We now prove the reflection statement. Let $f$ be a morphism in $\CDGC_1$ such that $f^*$ is a quasi-isomorphism. Again, %
we have that $H^n(f^*)\cong H_n(f)^*$. So $H_n(f)^*$ is an isomorphism for every $n$. This implies that $H_n(f)$ is an isomorphism for every $n$,  since linear duality of vector spaces reflects isomorphisms (linear duality is a faithful functor, %
so it reflects epimorphisms and monomorphisms). %
Therefore, $f$ is a quasi-isomorphism.
\end{proof}

We will now prove the above adjunction is a partial Quillen equivalence on the full subcategories of weak finite type objects. We recall that definition from \cite[Def. 3.2]{Pepe-Bruno}:

\begin{definition} \label{def-partial}
    Let $L:\C\rightleftarrows \D:R$ be a Quillen adjunction.
    Suppose $\C_0\subset \C$ and $\D_0\subset \D$ are full subcategories such that the following hold:
    \begin{enumerate}
    \item Every $X\in \C_0$ admits a cofibrant replacement $q_X:QX\to X$ with
$QX\in \C_0$, and every $Y\in \D_0$ admits a fibrant replacement
$p_Y:Y\to PY$ with $PY\in \D_0$,
        \item If $X\in \C_0$ is cofibrant, then $LX\in \D_0$, and if $Y\in \D_0$ is fibrant, then $RY\in \C_0$, 
        \item If $X\in \C_0$ is cofibrant, then the derived unit $X\xrightarrow{\eta_X} RLX \xrightarrow{Rp_{LX}} RPLX$ is a weak equivalence, where $p_{LX}:LX\to PLX$ is a fibrant replacement, and $\eta$ is the unit of $L\dashv R$. 
        Dually, if $Y\in \D_0$ is fibrant, then the derived counit $LQRY\xrightarrow{Lq_{RY}} LRY \xrightarrow{\epsilon_Y} Y$ is a weak equivalence, where $q_{RY}:QRY\to RY$ is a cofibrant replacement and $\epsilon$ is the counit of $L\dashv R$.
    \end{enumerate}
    Then we say $L:\C\rightleftarrows \D:R$ is a \emph{partial Quillen equivalence on the subcategories $\C_0$ and $\D_0$}.
\end{definition}

\begin{theorem} \label{thm-dual-partial} The adjunction $(-)^*:\CDGC_1\rightleftarrows (\CDGA^1)^\op:(-)^\circ$ is a partial Quillen equivalence on the full subcategories of the weak finite type connected cdgcs or cdgas.
\begin{proof} We have already verified the adjunction is Quillen in \cref{prop-dual-quillen}. We now verify the conditions of a partial Quillen equivalence.

Condition (1) is satisfied because both subcategories are closed under weak equivalences.

As for Condition (2): let $C\in \CDGC_1$ be weak finite type, %
so its minimal model $C\xrightarrow{\sim} M$ is of finite type. By \cref{prop-dual-ft-minimal}, $M^*$ is minimal and finite-type, so by \cref{prop-dual-quillen} %
and uniqueness of minimal models it is the minimal model of $C^*$. Thus, $C^*$ is weakly of finite type. Now let $A$ be a weak finite type cofibrant cdga. Let $m:M\xrightarrow{\sim} A$ be a finite-type minimal model. Since both $A$ and $M$ are cofibrant, the right Quillen functor $(-)^\circ$ takes $m$ to a quasi-isomorphism $A^\circ \to M^\circ=M^*$. By \cref{prop-dual-ft-minimal}, $M^*$ is a finite-type minimal cdgc, proving $A^\circ$ is weak finite type.

We now verify Condition (3). Let $A\in \CDGA^1$ be cofibrant and weak finite type. Since in $\CDGC_1$ every object is cofibrant, the derived counit of $(-)^*\dashv (-)^\circ$ at $A$ is simply the counit $\epsilon_A:A\to (A^\circ)^*$. Let $m:M\to A$ be a finite-type minimal model of $A$. Naturality of the counit gives a commutative square
\[\begin{tikzcd}
	M & A \\
	{(M^\circ)^*} & {(A^\circ)^*}
	\arrow["m", "\sim"', from=1-1, to=1-2]
	\arrow["{\epsilon_M}"', "\cong", from=1-1, to=2-1]
	\arrow["{\epsilon_A}", from=1-2, to=2-2]
	\arrow["{(m^\circ)^*}"', "\sim", from=2-1, to=2-2]
\end{tikzcd}\]
Since $m$ is a weak equivalence between cofibrant objects, then $(m^\circ)^*$ is a weak equivalence. Moreover, $\epsilon_M$ is an isomorphism: it is the canonical isomorphism to the double dual, $m\mapsto (\alpha\mapsto \alpha(m))$. By 2-out-of-3, the map $\epsilon_A$ is a weak equivalence as well.

Now let $C\in \CDGC_1$ be weak finite type. Let $q:B\to C^*$ be a cofibrant replacement of $C^*$; since $C^*$ is weak finite type, so is $B$. We want to prove that the composition $C\xrightarrow{\eta_C}(C^*)^\circ \xrightarrow{q^\circ} B^\circ$ is a weak equivalence, where $\eta_C$ denotes the unit of $(-)^* \dashv (-)^\circ$ at $C$. Consider the following commutative triangle: %
\[\xymatrix{B \ar[r]^-{\epsilon_B} \ar[rd]_-q & (B^\circ)^* \ar[d]^-{(q^\circ \circ \eta_C)^*} \\ & C^*.}\]
The counit $\epsilon_B$ is a weak equivalence by what we proved in the previous paragraph. Since $q$ is also a weak equivalence, by 2-out-of-3 the map $(q^\circ \circ \eta_C)^*$ is a weak equivalence. But linear duality reflects weak equivalences (\cref{prop-dual-quillen}), so $q^\circ \circ \eta_C$ is also a weak equivalence, as we wanted to show.
\end{proof}
\end{theorem}

\section{Crossed derivable adjunctions} \label{sect-crossed}

In this section, we assume our model categories to be endowed with chosen functorial (co)fibrant replacements, as in \cite{Hovey-ModelCategories}. This will make the proof of \cref{prop-crossed-derived-adj} more streamlined. %

If $L:\C\rightleftarrows \D:R$ is a Quillen adjunction, then $L$ (resp. $R$) preserves cofibrant (resp. fibrant) objects and weak equivalences between them. In general, this is not enough for the adjunction to be Quillen, %
but it is enough for some purposes. %
Following Cisinski's comment in \cite{Cisinski-MO}, we call such adjunctions \emph{derivable}. 
We now define a ``crossed'' (and partial) variant on this definition which will be useful for us in \cref{sect-Lie}, and prove some elementary properties these adjunctions satisfy.

\begin{definition} \label{def-adj-crossed} Let $L:\C\rightleftarrows \D:R$ be an adjunction between model categories. Let $\C_0\subset \C$, $\D_0\subset \D$ be full subcategories. Suppose the following conditions hold:
\begin{enumerate}
\item  $\C_0$, $\D_0$ are closed under weak equivalences, %
\item If $X\in \C_0$ is fibrant, then $L(X)\in \D_0$ and it is cofibrant; if $Y\in \D_0$ is cofibrant, then $R(Y)\in \C_0$ and it is fibrant,
\item $L$ takes weak equivalences between fibrant objects of $\C_0$ to weak equivalences, and $R$ takes weak equivalences between cofibrant objects of $\D_0$ to weak equivalences.
\end{enumerate}
Then we say $L\dashv R$ is a \emph{partial crossed derivable adjunction}. When $\C_0=\C$ and $\D_0=\D$, we say $L\dashv R$ is a \emph{crossed derivable adjunction}.
\end{definition}

\begin{remark}
In the above, we could do with less than $\C$ and $\D$ being model categories: the classes of (co)fibrations do not intervene other than in the (co)fibrant objects. However, since our guiding example for this definition in \cref{prop-LC-partial-mirror} does not need the added generality, we refrain from the generalization. Compare \cite[Sect. 2.2]{Riehl-Categorical}.
\end{remark}

If $\C_0\subset \C$ is a full subcategory of a model category, we denote by $h\C_0$ the full subcategory of $\Ho(\C)$ on the objects of $\C_0$. 

\begin{proposition} \label{prop-crossed-derived-adj}
    Let $L:\C\rightleftarrows \D:R$ be a partial crossed derivable adjunction. Then there is a derived adjunction
    \[\RL:h\C_0\rightleftarrows h\D_0:\LR,\]
    with $\RL(X)=L(X^\fib)$ and %
    $\LR(Y)=R(Y^\cof)$. %
    Here $(-)^\fib$ and $(-)^\cof$ denote the given functorial fibrant/cofibrant replacements in $\C$ and $\D$, respectively. The unit (resp. counit) of $\RL\dashv \LR$ is given on fibrant (resp. cofibrant) objects by the unit (resp. counit) of $L\dashv R$.
\end{proposition}
\begin{proof} Let $\C_{0,f}\subset \C$ denote the full subcategory of $\C$ on the fibrant objects of $\C_0$ and $\D_{0,c}\subset \D$ denote the full subcategory of $\D$ on the cofibrant objects of $\D_0$. By Condition (2), the adjunction $L\dashv R$ restricts to an adjunction $L:\C_{0,f}\rightleftarrows \D_{0,c}:R$. By Condition (3), these $L$ and $R$ preserve weak equivalences, so they induce an adjunction in localized categories $\overline{L}:\C_{0,f}[w^{-1}]\rightleftarrows \D_{0,c}[w^{-1}]:\overline{R}$, where $w$ denotes the class of weak equivalences.

Now let $i:\C_{0,f}\hookrightarrow \C_0$ denote the inclusion. It preserves weak equivalences, so it induces a functor $\C_{0,f}[w^{-1}]\to \C_0[w^{-1}]$. It is an equivalence of categories: its pseudo-inverse is given by $(-)^\fib$, which restricts to $\C_0$ by Condition (1). Similarly, we get an equivalence of categories $\D_{0,c}[w^{-1}]\simeq \D_0[w^{-1}]$.

Going further, we observe that $\C_0[w^{-1}]\simeq h\C_0$. Indeed, the standard proof that $\Ho(\C)\cong \C[w^{-1}]$ \cite[Thm. 1.2.10]{Hovey-ModelCategories} works within $\C_0$. Putting all this together, we obtain:
\[\begin{tikzcd}
	{h\C_0\simeq \C_0[w^{-1}] \simeq \C_{0,f}[w^{-1}]} & {\D_{0,c}[w^{-1}]\simeq \D_0[w^{-1}]\simeq h\D_0.}
	\arrow[""{name=0, anchor=center, inner sep=0}, "{\overline{L}}", from=1-1, to=1-2]
	\arrow[""{name=1, anchor=center, inner sep=0}, "{\overline{R}}", shift left=3, from=1-2, to=1-1]
	\arrow["\dashv"{anchor=center, rotate=-90}, draw=none, from=0, to=1]
\end{tikzcd}\]
Composing the functors, we obtain the desired derived adjunction. The assertion on the (co)unit is immediate, since the (co)unit of $\overline{L}\dashv \overline{R}$ is that of $L\dashv R$.
\end{proof}

\begin{remark} One could have guessed that (2) is not necessary in \cref{def-adj-crossed} %
for there to be a derived adjunction $\RL\dashv \LR$ as above. Indeed, in the non-partial, non-crossed case the analogous condition to (3) is enough: if $L\dashv R$ is an adjunction and $L$ (resp. $R$) preserves weak equivalences between cofibrant (resp. fibrant) objects, then the derived adjunction exists \cite[Cor., Pages 551-552]{Maltsiniotis}. %
However, in the crossed situation this is not enough, as the following simple example shows:

Consider the tensor-hom adjunction  \[L=\Q\otimes_\Z:\Ch(\Z)\rightleftarrows\Ch(\Z):\Hom_\Z(\Q,-)=R\] on unbounded chain complexes of abelian groups. %
Endow this category with the projective model structure \cite[Thm. 2.3.11]{Hovey-ModelCategories}. Now, since $\Q$ is flat then $L$ preserves all weak equivalences. On the other hand, if $P\in \Ch(\Z)$ is cofibrant, then each $P_n$ is a free abelian group \cite[Lem. 2.3.6]{Hovey-ModelCategories}.  Since $\Hom_\Z(\Q,P_n)=0$ (a free abelian group has no non-zero divisible subgroup), then $R(P)=0$. %
Therefore $R$ is zero on cofibrant objects, so it preserves weak equivalences between them.

Thus, the functors $\RL$ and $\LR$ exist, and they satisfy $\RL=\Q\otimes_\Z -:\D(\Z)\to \D(\Z)$, $\LR=0$. But $\RL$ is not left adjoint to $\LR$: compare $\Hom_{\D(\Z)}(\RL(\Q),\Q)\cong \Q$ with $\Hom_{\D(\Z)}(\Q,\LR(\Q))=0$. The same example shows $\RL$ is not right adjoint to $\LR$, either. This does not contradict the proposition, as this $L$ does not satisfy (2), e.g. $L(\Q)$ is not degree-wise free, so it is not cofibrant. %
\end{remark}

We now turn to (partial) crossed derivable equivalences.

\begin{proposition} \label{prop-crossed-partial-quillen-eq} Let $L:\C\rightleftarrows \D:R$ be a partial crossed derivable adjunction on the full subcategories $\C_0\subset \C$, $\D_0\subset \D$. The following are equivalent:
\begin{enumerate}
    \item[a)] If $X\in \C_0$ is fibrant, then the unit $\eta_X:X\to RLX$ is a weak equivalence, and if $Y\in \D_0$ is cofibrant, then the counit $\epsilon_Y:LRY\to Y$ is a weak equivalence. 
    \item[b)] A map $LX\to Y$ is a weak equivalence if and only if its transpose $X\to RY$ is a weak equivalence, for all $X\in \C_0$ fibrant and $Y\in \D_0$ cofibrant,
        \item[c)] The adjunction $\RL:h\C_0\rightleftarrows h\D_0:\LR$ is an adjoint equivalence.
\end{enumerate}
\begin{proof}
Since the (co)unit of $\RL\dashv \LR$ is that of $L\dashv R$ on (co)fibrant objects, the proof is essentially the same as the analogous statement characterizing adjoint equivalences of ordinary categories. %
\end{proof}
\end{proposition}

\begin{definition} In the situation of \cref{prop-crossed-partial-quillen-eq}, we say $L\dashv R$ is a \emph{partial crossed derivable equivalence on the subcategories $\C_0$ and $\D_0$}. When $\C_0=\C$ and $\D_0=\D$, we speak of a \emph{crossed derivable equivalence}. 
\end{definition}

\section{\texorpdfstring{Equivalences of $\infty$-categories}{Equivalences of infty-categories}}

As is well-known, a Quillen equivalence of model categories is more than an equivalence of their homotopy categories: it is an equivalence of homotopy theories, i.e. of $(\infty,1)$-categories. 
This can be made meaningful in several ways, according to how one interprets ``($\infty$,1)-category''. Interpreting it as an $\infty$-category (in the sense of Joyal and Lurie), this result follows from \cite[Thm. 2.1]{Mazel-Gee}.

We will now prove this is also true for partial Quillen equivalences and for partial crossed derivable equivalences. %

Recall that a \emph{relative category} is a pair $(\C,w)$ where $\C$ is a category and $w\subset \C$ is a wide subcategory whose morphisms are called \emph{weak equivalences}. We let $\C^\infty$ be the $\infty$-categorical localization of the nerve of $\C$ at $w$, as in \cite[Def. 1.3.4.1, Rmk. 1.3.4.2]{Lurie-HA}: this is an $\infty$-category. It comes with a functor $q:N(\C)\to \C^\infty$ with the following universal property: for every $\infty$-category $\D$, precomposition induces an equivalence
\begin{equation} \label{eq-univ-loc} q^*: \Fun(\C^\infty, \D)\xrightarrow{\sim} \Fun_w(N(\C), \D)\end{equation}
where $\Fun_w$ denotes the full subcategory on functors that send weak equivalences to equivalences. %

\begin{lemma} \label{lem-relcat-equiv}
Let $\A$ and $\B$ be relative categories and $L:\A\rightleftarrows \B: R$ %
be two functors that preserve weak equivalences. Suppose we are given natural transformations $\eta:\id_\A\Rightarrow RL$, $\epsilon:LR\Rightarrow \id_\B$ which are weak equivalences component-wise. Then there are induced functors $L^\infty:\A^\infty\rightleftarrows \B^\infty:R^\infty$ and they are mutually inverse equivalences of $\infty$-categories.
\end{lemma}

A similar result appears as early as \cite[Cor. 3.6]{Dwyer-Kan-Calculating}. There, $\eta$ and $\epsilon$ are assumed to be the unit and the counit of an adjunction $L\dashv R$, and the result is a Dwyer--Kan equivalence of simplicial categories.

\begin{proof}
    Since $L$ preserves weak equivalences, the composite $N(\A) \xrightarrow{N(L)} N(\B) \xrightarrow{q} \B^\infty$ sends weak equivalences to equivalences. By the universal property of localization, it factors through a functor $L^\infty:\A^\infty\to \B^\infty$, with a natural equivalence $L^\infty \circ q \simeq q\circ N(L)$. Similarly, we get an $R^\infty:\B^\infty\to \A^\infty$ such that $R^\infty \circ q \simeq q \circ N(R)$.

    Now apply the nerve functor to $\eta$, so we obtain a natural transformation $\id\Rightarrow N(RL)$ \cite[01DJ]{Kerodon}. Composing with $q$ gives a natural transformation $\widetilde \eta:q\Rightarrow qN(RL)\simeq R^\infty L^\infty q$. Every component of it is an equivalence, since $\eta$ is component-wise a weak equivalence. Therefore,  it is an equivalence (in its functor $\infty$-category), by \cite[01DK]{Kerodon}.

    The universal property of localization (\ref{eq-univ-loc}) applied to $\D=\A^\infty$ implies \[q^*:\Fun(\A^\infty, \A^\infty)\to \Fun(N(\A), \A^\infty)\] is fully faithful. In particular, $\widetilde\eta$ lifts to an $\eta^\infty:\id_{\A^\infty}\Rightarrow R^\infty L^\infty$, which is an equivalence since fully faithful functors reflect equivalences.

    We analogously get an $\epsilon^\infty:L^\infty R^\infty \Rightarrow \id_{\B^\infty}$, so $L^\infty$ and $R^\infty$ are mutually inverse equivalences.
\end{proof}

    If $\C$ is a model category and $\C_0\subset \C$ is a full subcategory not necessarily closed under weak equivalences (as can be the case in a partial Quillen equivalence, but not in a crossed one), then we will denote by $\C_{0,\infty}$  the full subcategory of $\C^\infty$ on the objects of $\C_0$. This is equivalent to $\C_0^\infty$, the $\infty$-localization of $\C_0$ at the weak equivalences, when $\C_0$ is closed under them (see for instance \cite[Sect. 6 and 8]{Dwyer-Kan-Calculating} in the setting of simplicial categories), %
    but in general they may differ.

\begin{proposition} \label{prop-infty}
    Let $L:\C\rightleftarrows \D:R$ be a partial Quillen equivalence or a partial crossed derivable equivalence on the full subcategories $\C_0\subset \C$, $\D_0\subset \D$. There is an induced equivalence of $\infty$-categories $\C_{0,\infty}\simeq \D_{0,\infty}$. %
   
    \begin{proof}
$\bullet$ We first treat the crossed case. Consider the relative categories $\C_{0,f}$ of fibrant objects in $\C_0$ with the weak equivalences of $\C$ and $\D_{0,c}$ of cofibrant objects in $\D_0$ with the weak equivalences of $\D$. The definition of a partial crossed derivable equivalence implies we have an adjunction $L:\C_{0,f} \rightleftarrows \D_{0,c}:R$ satisfying the hypotheses of \cref{lem-relcat-equiv}. So we get that $L^\infty:\C_{0,f}^\infty\rightleftarrows \D_{0,c}^\infty:R^\infty$ are mutually inverse equivalences of $\infty$-categories.

By \cite[Rmk. 1.3.4.16]{Lurie-HA}, $\C_{0,f}^\infty\simeq \C_{0,\infty}$ and $\D_{0,c}^\infty\simeq \D_{0,\infty}$, induced by the functorial (co)fibrant replacements of $\C$ and $\D$. Putting this together with the above furnishes the desired equivalence of $\infty$-categories. 

         $\bullet$ %
         We now treat the non-crossed case. Since $L\dashv R$ is a Quillen adjunction, it induces an adjunction of underlying $\infty$-categories $\C^\infty \rightleftarrows \D^\infty$ \cite[Thm. 2.1]{Mazel-Gee}. A technical point: Mazel-Gee's underlying $\infty$-category to $\C$ is different from what we denote $\C^\infty$. However, it is equivalent to $\C^\infty$, by \cite[After Def. 1.3.1]{Hinich-Dwyer-Kan}.

         The conditions of a partial Quillen equivalence imply that this adjunction restricts to $\C_{0,\infty}$ and $\D_{0,\infty}$, where it is an equivalence because it induces an equivalence on homotopy categories. %

         The fact that $\C_0$ and $\D_0$ may not be closed under weak equivalences is not an issue, since we can pass to their closure under weak equivalences without affecting the argument \cite[Rmk. 1.2.10.2]{LurieHTT}.
    \end{proof}
\end{proposition}

\begin{remark} \label{rmk-mapping-spaces}
    In particular, if $L:\C\rightleftarrows \D:R$ is a partial crossed derivable equivalence on the full subcategories $\C_0\subset \C$, $\D_0\subset \D$, then there are weak equivalences of mapping spaces
    \[\Map_{\C_{0,\infty}}(X,Y) \simeq \Map_{\D_{0,\infty}}(L(X^\fib), L(Y^\fib)).\] %
    If $L:\C\rightleftarrows \D:R$ is a partial Quillen equivalence on $\C_0$ and $\D_0$, we get weak equivalences of mapping spaces,     where $X^\cof, Y^\cof\in \C_0$:
    \[\Map_{\C_{0,\infty}}(X,Y) \simeq \Map_{\D_{0,\infty}}(L(X^\cof), L(Y^\cof)).\]

    Taking $\pi_0$ of such weak equivalences recovers the corresponding bijections of morphism sets in the homotopy categories.
\end{remark}

\begin{remark} Applying \cref{prop-infty} to the partial Quillen equivalence in  \cref{thm-dual-partial} gives a contravariant equivalence of $\infty$-categories between the weak finite type connected cdgcs and cdgas. Compare this with \cite[Thm. 2.26]{Roca} for a similar equivalence whose proof does not adapt to our context since the corresponding dg operads here are not cofibrant.
\end{remark}

\section{Lie algebras} \label{sect-Lie}

We let $\DGL$ denote the category of dg Lie algebras (dgls) in $\Chaz$. We denote by %
$\LL(V)$ the free graded Lie algebra on a non-negatively graded vector space $V$. %
A dgl $L$ is \emph{minimal} \cite[Page 437]{Neisendorfer} if its underlying graded Lie algebra is free and $dL\subset [L,L]$. %
Let $\Gamma^1L:=L$, $\Gamma^{r+1}L:=[L,\Gamma^rL]$ for $r\geq 1$. 
A graded Lie algebra $L$ is \emph{nilpotent} if it is degreewise nilpotent, i.e. the  descending central series $\Gamma^1L\supset  \Gamma^2L\supset \cdots$ terminates at zero in each degree. %
For a degree $q$, let $n(q)$ be the least $r\geq 0$ such that $(\Gamma^{r+1}L)_q$ is zero.  The supremum of $\{n(q)\}_{q\geq 0}$ is the \emph{nilpotency degree} of $L$ (it may be infinite even for nilpotent $L$).

\subsection{Extensions, nilpotency and minimality}

If $0\to A\to E\xrightarrow{p} N\to 0$ is a short exact sequence of graded Lie algebras, we say $E$ is an \emph{extension} of $N$ by $A$. If the kernel of $p$ is in the center of $E$, i.e. $\ker(p)\subset Z(E)$, we say the extension is \emph{central}, in which case $A$ is abelian.

If $C$ is a semi-cofree cdgc then $s^{-1}H_*PC$ becomes a graded Lie algebra in a natural way \cite[Prop. 2.1]{Neisendorfer}. %
In particular, if moreover $d_{PC}=0$ (which is the case for a minimal cdgc), then $s^{-1}PC$ is a graded Lie algebra.

\begin{lemma} \label{lem-central}
\begin{enumerate}
    \item 
    Let $C'\to C$ be a map in $\CDGC_1$ of semi-cofree cdgcs with $d_{PC}=0=d_{PC'}$ and let $W$ be a graded vector space concentrated in a single degree $\geq 1$. Then  $C'\to C$ is a non-primitive elementary extension with vector space of generators $W$ if and only if there is a short exact sequence \[0\to s^{-1}W \to s^{-1}PC'\to s^{-1}PC\to 0\] which is a central extension of graded Lie algebras. %
    \item A central extension $E$ of a nilpotent graded Lie algebra  $N$ is nilpotent. If the nilpotency degree of $N$ is at most $n$, then the nilpotency degree of $E$ is at most $n+1$.
    \item A graded Lie algebra $L$ obtained from 0 by finitely many central extensions is nilpotent. If there are $r$ extensions, then the nilpotency degree of $L$ is at most $r$.
        \end{enumerate}
\end{lemma}

\begin{proof}
    \begin{enumerate}\item This is a particular case of \cite[Prop. 3.8]{Neisendorfer}. %
    \item
  Suppose
    \[\xymatrix{0\ar[r] & A \ar[r] & E \ar[r]^-p & N \ar[r] & 0}\]
is a central extension of graded Lie algebras. 
Fix a degree $q\geq 0$ and suppose  $c\geq 0$ is such that $(\Gamma^{c+1}N)_j=0$ for all $0\leq j \leq q$. Then $p((\Gamma^{c+1}E)_j)\subset(\Gamma^{c+1}N)_j=0$, so 
$(\Gamma^{c+1}E)_j\subset \ker(p)_j\subset   Z(E)_j$ for every such $j$. Thus,
    \[(\Gamma^{c+2}E)_q=\sum_{i+j=q}[E_i,(\Gamma^{c+1}E)_j]=0.\]
    This proves the two statements.
\item What the hypothesis means is that there is a sequence of graded Lie algebras
\[\xymatrix{0=K_0 & K_1 \ar[l] & K_2 \ar[l] & \ar[l] \cdots & K_r \ar[l]}\]
where every map is degreewise surjective with central kernel. Since $\Gamma^1K_0=0$, the previous point gives inductively $\Gamma^{i+1}K_i=0$ for $0\leq i\leq r$.
Thus $\Gamma^{r+1}K_r=0$: the final graded Lie algebra is nilpotent of nilpotency degree at most $r$. \qedhere
    \end{enumerate}
\end{proof}

\begin{lemma} \label{lem-minimal-nilpotent} Let $C\in \CDGC_1$.
    \begin{enumerate}
    \item   If $C$ is minimal and finite-type, then the graded Lie algebra $s^{-1}PC$ is nilpotent and of finite type.
    \item If $C$ is semi-cofree, $d_{PC}=0$ and the graded Lie algebra $s^{-1}PC$ is nilpotent, then $C$ is minimal. %
    \item If $C$ is semi-cofree %
    and the graded Lie algebra $s^{-1}H_*PC$ is nilpotent, then there exist a minimal cdgc $M$, a graded vector space $U$ concentrated in degrees $\geq 2$ and an isomorphism $C\cong M\otimes D^c(U)$ of cdgcs. In particular, $C$ is fibrant.
        \end{enumerate}

    \begin{proof}
     (1) Consider a tower for $C$ as in \cref{prop-cdgc-ft-minimal}: $\Q=C_0\leftarrow C_1\leftarrow \cdots$. Fix $n\geq 0$. There are finitely many elementary extensions in this tower which are of degree $\leq n+1$, and they come first in the tower.  Let $C_N$ be the stage immediately after the last of these, so $PC_N$ is concentrated in degrees $\leq n+1$ and the subsequent maps induce isomorphisms on primitives in those degrees. Since $P$ preserves limits, the projection $C\to C_N$ induces an isomorphism on primitives in degrees $\leq n+1$. Letting $K=s^{-1}PC$ and $K_i=s^{-1}PC_i$, we get a finite tower of graded Lie algebras
    \[\xymatrix{0=K_0 & K_1\ar[l] & \cdots\ar[l] & K_N\ar[l]}\]
    such that the induced map $K\to K_N$ is an isomorphism in degrees $\leq n$, while $(K_N)_{>n}=0$. Thus its kernel is $K_{>n}$, so $K_N\cong K_{\leq n}$.
    
    By \cref{lem-central}(1), the tower displays $K_{\leq n}$ as obtained from 0 by finitely many central extensions. Therefore, by \cref{lem-central}(3), $K_{\leq n}$ is nilpotent. Since $K\to K_{\leq n}$ is an isomorphism in degrees $\leq n$, this implies
$(\Gamma^rK)_n=0$ for sufficiently large $r$. 
This holds for every $n$, thus $K$ is nilpotent.  \medskip

(2) By \cref{prop-cdgc-ft-minimal}, it suffices to prove  $C=(\Lambda^c(V),d)$ is constructible by non-primitive elementary extensions starting from $\Q$. Let $K=s^{-1}PC$. Let  $V_{n,r}=s(K_{<n}\oplus K_n/(\Gamma^r K)_n)$ and  $C_{n,r}=\Lambda^c(V_{n,r})$ for $n\geq 0$ and $r\geq 1$. We have a commutative diagram of projection maps:
\[\begin{tikzcd}
	C & V \\
	{C_{n,r}} & {V_{n,r}}
	\arrow["\pi", from=1-1, to=1-2]
	\arrow["{Q_{n,r}}"', from=1-1, to=2-1]
	\arrow["{p_{n,r}}", from=1-2, to=2-2]
	\arrow["{\pi_{n,r}}"', from=2-1, to=2-2]
\end{tikzcd}\]%

\underline{Differential on $C_{n,r}$}: We will induce a coderivation $d_{n,r}$ on $C_{n,r}$ from the differential on $C$ such that $Q_{n,r} d = d_{n,r}Q_{n,r}$, by means of \cite[Lem. 1.7]{Neisendorfer}: it is enough to define a graded linear map $e:C_{n,r}\to V_{n,r}$ of degree $-1$ such that $\pi_{n,r} Q_{n,r}d=eQ_{n,r}$. Both  $Q_{n,r}$ and $\pi_{n,r}$ are surjective, so it is enough to prove that $\pi_{n,r} Q_{n,r}d(\ker Q_{n,r})=0$. %
By the commutativity of the square above, we need to observe that $p_{n,r} \pi d(\ker Q_{n,r})=0$.

Since $d_{PC}=d|_V=0$ and the elements in the image of $p_{n,r}$ have degrees at most $n+1$, a word in $\ker Q_{n,r}$ can have a non-zero image by $p_{n,r}\pi d$ only if it has two factors, one in $s(\Gamma^rK)_n$ and the other in $sK_0$. On these words, $\pi d$ encodes the bracket, whose image lies in
\[
s[(\Gamma^rK)_n,K_0] \subset s(\Gamma^{r+1}K)_n \subset \ker p_{n,r},
\]
so $p_{n,r}\pi d(\ker Q_{n,r})=0$, as we wanted.

Since $d_{n,r}^2Q_{n,r}=Q_{n,r}d^2=0$ and $Q_{n,r}$ is surjective, we have $d_{n,r}^2=0$. %
Thus, $(C_{n,r}, d_{n,r})$ is a cdgc and $Q_{n,r}$ is a cdgc morphism, so in particular $d_{PC_{n,r}}=0$.
\medskip

\underline{Definition of the tower}: The quotient maps $K_n/(\Gamma^{r+1}K)_n\to K_n/(\Gamma^rK)_n$ induce maps $V_{n,r+1}\to V_{n,r}$ and hence maps $c_{n,r}:C_{n,r+1}\to C_{n,r}$ of graded coalgebras. By construction, they satisfy $c_{n,r}Q_{n,r+1}=Q_{n,r}$.  This identity plus surjectivity of $Q_{n,r+1}$ proves the maps $c_{n,r}$ are cdgc maps.

Since $K$ is nilpotent, for every $n\geq 0$ there exists an $R_n$ such that $(\Gamma^{R_n}K)_n=0$. 
Moreover, $C_{0,1}=\mathbb Q$ and $C_{n,R_n}=C_{n+1,1}$ as cdgcs. %
Thus, we may concatenate the sequences \[C_{n,1}\leftarrow C_{n,2} \leftarrow\cdots\leftarrow C_{n,R_n}\] to get an inverse tower \[\Q=C_0\leftarrow C_1\leftarrow \cdots\]
which is, by construction, degree-wise eventually constant. The maps $Q_{n,r}$ form a cone for this tower such that for every $q$, $Q_{n,r}$ is an isomorphism in degrees $\leq q$ for all $r$ and all $n$ sufficiently large. Therefore, by \cref{lem-tower}, the induced map $C\to \varprojlim_j C_j$ is an isomorphism.

\medskip

\underline{$c_{n,r}$ is a non-primitive elementary extension}: We use \cref{lem-central}(1) to prove that $c_{n,r}:C_{n,r+1}\to C_{n,r}$ is a non-primitive elementary extension of cdgcs. We already observed that $d_{PC_{n,r}}=0$, so it suffices to prove that the kernel of the map of graded Lie algebras $s^{-1}Pc_{n,r}:s^{-1}PC_{n,r+1}  \to s^{-1}PC_{n,r}$ is central. The cdgc map $Q_{n,r}$ induces an isomorphism of graded Lie algebras %
\[s^{-1}PC_{n,r}=K_{<n}\oplus K_n/(\Gamma^rK)_n \cong K/(K_{>n}\oplus (\Gamma^rK)_n),\] so we need to prove that the following short exact sequence of graded Lie algebras is central:

\[\begin{tikzcd}
	0 & {(\Gamma^rK)_n/(\Gamma^{r+1}K)_n} & {K/(K_{>n}\oplus (\Gamma^{r+1}K)_n)} & {K/(K_{>n}\oplus (\Gamma^rK)_n)} & 0
	\arrow[from=1-1, to=1-2]
	\arrow[from=1-2, to=1-3]
	\arrow[from=1-3, to=1-4]
	\arrow[from=1-4, to=1-5]
\end{tikzcd}\]

Thus, we need to prove that $[(\Gamma^rK)_n/(\Gamma^{r+1}K)_n, K/(K_{>n}\oplus(\Gamma^{r+1}K)_n)]=0$. For $a\in (\Gamma^rK)_n$ and $x\in K_m$, we want $[x,a]=0$ in $K/(K_{>n}\oplus(\Gamma^{r+1}K)_n)$. For $m=0$, $[x,a]\in[K_0,(\Gamma^rK)_n]\subset(\Gamma^{r+1}K)_n$; for $m>0$, the bracket has degree $m+n>n$, so $[x,a]\in K_{m+n}\subset K_{>n}$ and we are done.

\medskip

(3) We argue similarly as in \cite[Prop. 3.12]{Neisendorfer}. Suppose the underlying graded coalgebra of $C$ is $(\Lambda^c(W),d)$, so $PC$ is a chain complex $(W,d|_W)$ concentrated in degrees $\geq 1$. 

We will use the following notation: if $V$ is a positively graded vector space, then $\DD(V)=s^{-1}V\oplus V$ is the chain complex with $|s^{-1}v|=|v|-1$, $d(v)=s^{-1}v$, $d(s^{-1}v)=0$.
 \medskip

\underline{Decomposition $W=(H,0)\oplus \DD(U)$}: For each $n$, consider the $n$-cycles of $W$: $Z_n\subset W_n$ and choose a vector space complement, so $W_n=Z_n\oplus U_n$. Similarly, consider the $n$-boundaries $B_n\subset Z_n$ and choose a complement, so $Z_n=H_n\oplus B_n$. Combining these decompositions, we get $W_n=H_n\oplus B_n\oplus U_n$, with the differential $d|_W:U_{n+1}\to B_n$ an isomorphism. Thus, $W\cong (H,0)\oplus \DD(U)$ as chain complexes, and the inclusion $(H,0)\hookrightarrow W$ is a quasi-isomorphism: $H\cong H_*W$. Since $W$ is concentrated in degrees $\geq 1$, then $Z_1=W_1$, so $U_1=0$ and thus $U$ is concentrated in degrees $\geq 2$. \medskip

\underline{Definition of $\rho:C\to D^c(U)$}: We define a cdgc map $\rho:C\to D^c(U)$. Since $D^c(U)$ is the cofree cdgc on the chain complex $\DD(U)$, by adjunction it suffices to define a chain map $\overline{C}\to \DD(U)$. Now note that $\DD(U)$ is an  injective object of $\Chaz$: indeed, $\DD(U)=\prod_{n\geq 2}\DD^n(U_n)$ where $\DD^n(V)$ for $V$ a vector space denotes $V\xrightarrow{\id} V$ concentrated in degrees $n$ and $n-1$. It suffices to observe that $\DD^n(V)$ is injective. %
But this is true because $\DD^n:\Vect\to \Chaz$ is right adjoint to evaluation at $n-1$, and every vector space is injective. %

Consider the inclusion $W=PC\hookrightarrow \overline{C}$. Injectivity of $\DD(U)$ provides the desired extension of the projection $W\to\DD(U)$ to a chain map $\overline{C}\to \DD(U)$.\medskip

\underline{Definition of $\gamma:C\to \Lambda^c(H)$}: We define a graded coalgebra map $\gamma: C\to \Lambda^c(H)$. By adjunction, it suffices to define a graded linear map $\overline{C}\to H$. Since $H$ is injective (as is every graded vector space), we obtain such a map as an extension of the projection $W\to H$ along the inclusion $W\hookrightarrow \overline{C}$. \medskip

\underline{Lemma}: A morphism of graded coalgebras $f:\Lambda^c(V)\to D$ is an isomorphism as soon as $P(f)$ is an isomorphism. To see this, first recall that for a general morphism of cdgcs $f$, if $P(f)$ is injective then $f$ is injective \cite[Lem. 3.2]{Neisendorfer}. Now observe that  $\Lambda^c(V)$ is injective in the category of graded coalgebras: indeed, if $A\hookrightarrow B$ is an injection and $A\to \Lambda^c(V)$ is a morphism, then by adjunction it is enough to lift the transpose $\overline{A}\to V$ along the inclusion $\overline{A}\hookrightarrow \overline{B}$, which can always be done in graded vector spaces. To finish, apply to $P$ the following remark: a functor $F$ that reflects monomorphisms also reflects isomorphisms $f:E\to D$ whose domain is injective. We prove this: since $F(f)$ is an isomorphism, it is a monomorphism, so $f$ is a monomorphism with injective domain. It thus admits a retraction $r:D\to E$: $rf=\id$. Applying $F$ shows $Fr=(Ff)^{-1}$, thus $Fr$ is a isomorphism, in particular a monomorphism and thus $r$ is a monomorphism. Since $r(fr)=r$, this shows $fr=\id$ and thus $f$ is an isomorphism. \medskip

\underline{Definition of $\Phi:C\xrightarrow{\cong} \Lambda^c(H)\otimes D^c(U)$}: Since the tensor product is the product in the category of graded coalgebras, we assemble $\rho$ and $\gamma$ into a graded coalgebra map \[\Phi:C\to \Lambda^c(H)\otimes D^c(U)\cong \Lambda^c(H\oplus s^{-1}U\oplus U).\]
On primitives, this map is the isomorphism $W\cong H\oplus s^{-1}U \oplus U$, so by the lemma above, $\Phi$ is an isomorphism of graded coalgebras. \medskip

\underline{Definition and minimality of $M$}: Let $(M, d_M)$ be the sub-cdgc of $C$ given as the categorical kernel of $\rho$ in $\CDGC_1$. Since $\Q$ is the zero object of $\CDGC_1$, this means we form the following pullback in $\CDGC_1$, where $\eta$ denotes the coaugmentation:
\[\begin{tikzcd}
	M & \Q \\
	C & {D^c(U)}
	\arrow[from=1-1, to=1-2]
	\arrow["i"', from=1-1, to=2-1]
	\arrow["\lrcorner"{anchor=center, pos=0.125}, draw=none, from=1-1, to=2-2]
	\arrow["\eta", from=1-2, to=2-2]
	\arrow["\rho"', from=2-1, to=2-2]
\end{tikzcd}\]
Since $\DD(U)$ is the complement of $H_*W$ in $W$, intuitively this construction discards the homologically irrelevant information in $C$ coming from $U$. We will prove $M$ is minimal using (2) which we just proved.

The tensor product is the product in $\CDGC_1$. Consider the product projection $\pi:\Lambda^c(H)\otimes D^c(U)\to D^c(U)$, which satisfies $\pi\circ \Phi=\rho$. Since $\Phi$ is an isomorphism of graded coalgebras, we get $M\cong \ker(\pi)\cong \Lambda^c(H)$
as graded coalgebras, so $M$ is semi-cofree.

Since $P$ preserves pullbacks, applying $P$ to the pullback diagram defining $M$ proves $PM=\ker(W\to \DD(U))=(H,0)$ %
, so $d_{PM}=0$.

Finally, we prove $s^{-1}PM$ is nilpotent.  The morphism of cdgcs $i:M\to C$ induces a morphism of graded Lie algebras $H_*(s^{-1}P(i)):s^{-1}PM= s^{-1}H_*(PM)\to s^{-1}H_*(PC)$. But $P(i)$ is the inclusion $(H,0)\hookrightarrow W$ which is a quasi-isomorphism, as already noted in the first step of this proof. %
 So we get an isomorphism $s^{-1}PM\cong s^{-1}H_*(PC)$ of graded Lie algebras. The latter is nilpotent by hypothesis, so the former is nilpotent, too.

\medskip

\underline{The isomorphism $C\to M\otimes D^c(U)$:}

Since $i$ is a map between semi-cofree cdgcs and we just proved $P(i)$ is a quasi-isomorphism, then  $i$ itself is a quasi-isomorphism, by \cite[Lem. 1.10b]{Neisendorfer}. Thus $i$ is an acyclic cofibration. Fibrancy of $M$ provides a retraction $r:C\to M$,  $ri=\id$.

Consider the cdgc map $(r,\rho):C\to M\otimes D^c(U)$. %
It induces the chain map $PC\cong H\oplus \DD(U) \to H\oplus \DD(U)$, $(h,u)\mapsto (h+Pr(u),u)$  on primitives. It is an isomorphism with inverse $(h',u)\mapsto (h'-Pr(u),u)$, thus $(r,\rho)$ is an isomorphism by the underlined lemma above.

\medskip

Finally, we need to prove that $C$ is fibrant. Since $M$ is minimal, it is fibrant, and $D^c(U)$ is also fibrant \cite[Lem. 5.3a)]{Neisendorfer}. Since the tensor product is the product in $\CDGC_1$, %
 $M\otimes D^c(U)$ is fibrant and thus $C$ is fibrant, too.
    \end{proof}
\end{lemma}

\subsection{\texorpdfstring{The model structure on $\DGL$}{The model structure on DGL}}

Neisendorfer \cite[Prop. 5.2]{Neisendorfer} establishes a model structure on $\DGL$, where the weak equivalences are the quasi-isomorphisms. Every object is fibrant and the cofibrant objects are precisely the dgls with underlying free graded Lie algebra \cite[Prop. 5.7c, Lem. 1.3]{Neisendorfer}.  %

There is a small inaccuracy in \cite{Neisendorfer} regarding the fibrations in $\DGL$: they are declared to be the degreewise surjections, rather than the degreewise surjections \emph{in positive degrees}. That would have been the more natural analog of Quillen's fibrations (surjections in degrees $\geq 2$) from \cite[Thm. II.5.1]{Quillen-rht} in the model structure on dgls concentrated in positive degrees. 
The author even asserts \cite[p.445]{Neisendorfer}:  \textit{ Although Quillen [...] uses a slightly different notion of fibration, his proof that they form a closed model category remains valid in this more general context.}
In fact, such a model structure on $\DGL$ cannot exist. The only source we have found that sets Neisendorfer's fibrations straight is \cite[p.45]{Tanre83}, who does not write that Neisendorfer's fibrations are not exactly those. 
He does not prove that Neisendorfer's structure with his stated fibrations cannot exist, but this is not hard:

\begin{proposition} The model structure %
on $\DGL$ with quasi-isomorphisms as weak equivalences and degreewise surjections as fibrations 
does not exist. 
    \begin{proof}
First note that if $g:C\to B$ is a morphism in $\Chaz$ such that $g_0:C_0\to B_0$ is surjective, then $H_0(g):H_0(C)\to H_0(B)$ is surjective. %
Now, if the claimed model structure on $\DGL$ existed, 
then any morphism $f:A\to B$ would factor as an acyclic cofibration followed by a fibration: 
$A\stackrel{\sim}{\rightarrowtail} C\twoheadrightarrow B$.
Taking homology, this implies that $H_0(f):H_0(A)\to H_0(B)$ is surjective. 
This is not always true, of course: take $f$ to be $0\to \Q$,
where $\Q$ is the abelian 1-dimensional dg Lie algebra concentrated in degree 0.
\end{proof}
\end{proposition}

The above hiccup regarding the fibrations in $\DGL$ does not pose a problem: once they are taken to be surjective in positive degrees, everything else goes through. This is the model structure we consider on $\DGL$ from now on.  %

Note that a quasi-isomorphism $p:E\to B$ surjective in positive degrees is also surjective in degree zero. Indeed, given $b_0\in B_0$, lift its class in $H_0(B)$ to $e_0\in E_0$, so there exists a $b_1\in B_1$ such that $b_0-p(e_0)=d b_1$. %
Choose an $e_1\in E_1$ such that $p(e_1)=b_1$. Then $p(e_0+de_1)=b_0$. Thus, the acyclic fibrations are exactly the degreewise surjective quasi-isomorphisms.

\subsection{The adjunction between dgls and cdgcs}

There is an adjunction  $\L:\CDGC_1\rightleftarrows \DGL:\C$ \cite[Sect. B.6]{Quillen-rht}, see also \cite[Sect. 4]{Neisendorfer}, where $\C$ is the dg version of the Chevalley--Eilenberg chain complex for computing Lie algebra homology, and $\L$ outputs the dg Lie algebra of primitive elements in the cobar construction.  More explicitly: 
For a connected cdgc $C$, $\L(C)=(\LL(s^{-1}\overline{C}),d)$ where $d$ is determined from the differential and the coproduct of $C$. For a dgl $L$, $\C(L)=(\Lambda^c(sL),d)$ where $d$ is defined using the differential and the bracket of $L$.

\begin{lemma} \label{lem-without-nc}
    \begin{enumerate}
        \item If $C\in \CDGC_1$ is fibrant and weakly of finite type, then the graded Lie algebra $H_*\L(C)$ is nilpotent and finite-type.
        \item If $L\in \DGL$ is such that $H_*(L)$ is nilpotent and finite-type, then $\C(L)$ is fibrant and weakly of finite type.
    \end{enumerate}
    \begin{proof}
    \begin{enumerate}
    
    \item  Let $C\to M$ be a minimal model for $C$, which is finite-type by hypothesis. Since $C$ is fibrant, by \cite[Prop. 6.3b]{Neisendorfer}, we have an isomorphism of graded Lie algebras $H_*\L(C)\cong s^{-1}PM$, so $H_*\L(C)$ is finite-type and nilpotent by \cref{lem-minimal-nilpotent}(1).

\item By definition of $\C$, %
$s^{-1}H_*P\C(L)\cong H_*(L)$ as graded Lie algebras, so the former is finite-type and nilpotent.  By \cref{lem-minimal-nilpotent}(3), we have that $\C(L)$ is fibrant, and thus by \cref{lem-wft-cdgc} it is also weakly of finite type. \qedhere
        \end{enumerate}
    \end{proof}
\end{lemma}

\begin{theorem} \label{prop-LC-partial-mirror} The adjunction  $\L:\CDGC_1\rightleftarrows \DGL:\C$ is a partial crossed derivable equivalence on the connected cdgcs weakly of finite type and the dgls $L$ such that $H_*(L)$ is nilpotent and of finite type.
\end{theorem}
\begin{proof}
    We first prove this is a crossed derivable adjunction. We check the three conditions:
    
    \begin{enumerate}
    \item It is clear the subcategories are closed under weak equivalences.
        \item Since dgls in the image of $\L$ have a free underlying graded Lie algebra by definition, they are cofibrant. %
        Now suppose $C\in \CDGC_1$ is fibrant and weakly of finite type. Then $H_*\L(C)$ is nilpotent and finite-type by \cref{lem-without-nc}(1).

        In the other direction, suppose $L\in \DGL$ is a cofibrant dgl such that $H_*(L)$ is nilpotent and finite-type. Then \cref{lem-without-nc}(2) gives that $\C(L)$ is fibrant and weakly of finite type; note cofibrancy of $L$ was not needed.
        
        \item That $\L$ preserves weak equivalences between fibrant cdgcs is \cite[Prop. 6.4]{Neisendorfer}, and that $\C$ preserves weak equivalences between dgls is \cite[Prop. 4.4a]{Neisendorfer}.
    \end{enumerate}
Finally, the components of the unit and counit of $\L\dashv \C$ are weak equivalences: this is \cite[Prop. 4.1]{Neisendorfer}, so $\L\dashv \C$ is indeed a partial crossed derivable equivalence on the specified subcategories.
\end{proof}

\begin{proposition} \label{rmk-LC-not-quillen} The adjunction $\L\dashv \C$ is \emph{not} a (non-crossed) derivable adjunction. In particular, it is not Quillen. 

\begin{proof} Since every object of $\CDGC_1$ is cofibrant, it suffices to show there is a weak equivalence in $\CDGC_1$ which is not mapped by $\L$ to a weak equivalence.

Let $L$ be the dgl $\Q\{a,b\}$, $|a|=|b|=0$ with zero differential and $[a,b]=b$.  Let $L'$ be the abelian dgl $\Q{a}$, $|a|=0$ and consider the dgl quotient map $\pi:L\to L'$. Let us apply $\C$ to it and see it is a weak equivalence.

Observe that \[C:=\C(L)=(\Lambda^c(sL),d)=(\Q\{1,e,f,w\},d),\] where $e=sa$, $f=sb$, $w=e\otimes f -f\otimes e$, %
so $|e|=|f|=1$, $|w|=2$. %
The coproduct on $\C(L)$ is given by
\[
\Delta(1)=1\otimes1, \qquad 
\Delta(e)=e\otimes1+1\otimes e,\qquad
\Delta(f)=f\otimes1+1\otimes f, \qquad
\Delta(w)=w\otimes1+1\otimes w+e\otimes f-f\otimes e.
\]
The differential is $de=df=0$, $dw=s[a,b]=f$. %
We let \[C':=\C(L')=(\Q\{1,e\},0).\] Applying $\C$ to $\pi:L\to L'$ gives the cdgc map $p:C\to C'$, $p(1)=1$, $p(e)=e$, $p(f)=p(w)=0$. Its kernel is the acyclic complex $(\cdots \to \Q{w}\underset{\cong}{\overset{d}{\to}} \Q{f}\to 0)$, so $p$ is a weak equivalence.

We now apply $\L$ to $p$. We have \[\L(C)=(\LL(A,B,U),d) \quad \text{and} \quad \L(C')=(\Q{A},0),\] where $A=s^{-1}e, B=s^{-1}f, U=s^{-1}w$, so $|A|=|B|=0$, $|U|=1$, with $dA=dB=0$, $dU=[A,B]-B$. %
We compute $H_0\L(p)$: first observe $H_0\L(C)$ is the quotient of $\LL(A,B)$ by the Lie ideal generated by $dU$, so it is isomorphic to $L$. Moreover, $H_0 \L(C')\cong L'$: under these identifications, $H_0\L(p)$ is the original map $\pi:L\to L'$, $a\mapsto a$, $b\mapsto 0$. This is not an isomorphism, so $\L(p)$ is not a quasi-isomorphism. 
\end{proof}
\end{proposition}

\begin{remark} \label{rmk-not-fibrant} Since $\L\dashv \C$ is, on the other hand, a partial crossed derivable adjunction, this means that one of $C$ or $C'$ in the proof above cannot be fibrant. We will prove $C$ is not fibrant by exploiting the following general observation: if $q:C\to B$ is a quasi-isomorphism in $\CDGC_1$ and $C$ is fibrant, then $H_*Pq$ is injective. To see this, factor $q$ as a cofibration $i$ followed by an acyclic fibration $p$. By two-out-of-three, $i$ is a quasi-isomorphism, so fibrancy of $C$ provides a retraction of $i$, proving $i$ is a split monomorphism. Thus $H_*(Pi)$ is injective, %
while $H_*(Pp)$ is an isomorphism because $P$ is right Quillen.

Apply this to a projection $q:C\to C/J$, where $J\subset C$ is an acyclic dg coideal. This proves that if $C$ is fibrant, then every primitive cycle which belongs to an acyclic dg coideal of $C$ is the boundary of a primitive element. %

Going back to the $C$ in \cref{rmk-LC-not-quillen}, take $J=\Q\{f,w\}$. It is an acyclic dg coideal and $f$ is  a primitive cycle, so if $C$ were fibrant, then $f$ would be the boundary of a primitive element. But $(PC)_2=0$ and $f\not=0$, so this is impossible.
\end{remark}

\begin{remark} In \cref{prop-LC-partial-mirror}, on the coalgebraic side we have a weak finite type hypothesis while on the dgl side we have ``homology nilpotent and finite-type''. One could then wonder whether this condition is equivalent to ``finite-type minimal model'', i.e. a weak finite type condition for dgls. However, this is not true, as the following simple example shows: %

Let $L$ be the %
abelian dgl $\Q^2$ concentrated in degree zero. So $H_*L=\Q^2$ is finite-type and nilpotent (it is abelian). %
We now prove no cofibrant (a fortiori, minimal) model of $L$ is finite-type, %
i.e. there are no weak equivalences $N\to L$ with $N$ cofibrant and of finite type. Indeed, if $N$ is cofibrant, then its underlying graded Lie algebra is free, so $N_0$ is an ordinary free Lie algebra. If we had a weak equivalence $N\to L$, then we would have an isomorphism $H_0(N)\cong H_0(L)=\Q^2$. On the other hand, we have a projection $N_0\to H_0(N)\cong \Q^2$. %
If $N_0$ were free on zero or one generator, then it would have dimension zero or one, giving  a contradiction with this surjection. Therefore, $N_0$ is a free Lie algebra on two generators or more. In particular, it is infinite-dimensional, so $N$ is not of finite type.
\end{remark}

\section{Models for spaces}

We now combine the results we have obtained and connect them with spaces in a structured way, improving \cite[Prop. 7.3]{Neisendorfer}, which gave only equivalences of homotopy categories.

In \cref{sect-Lie} we made the homotopical link between  $\DGL$ and $\CDGC_1$, in \cref{sect-duality} between $\CDGC_1$ and $\CDGA^1$ and in \cite{Pepe-Bruno}, together with Moreno--Fernández, between $\CDGA^1$ and the category of pointed simplicial sets $\sSet_*$ with the standard Kan--Quillen model structure. We now recall the latter link, providing additional details.

The full subcategory $\CDGA^1$ is reflective in $\CDGA_{/\Q}$. Explicitly, the reflector $L:\CDGA_{/\Q}\to \CDGA^1$ is given by taking the quotient by the dg ideal generated by the kernel of the augmentation in degree 0. 

\begin{lemma}\label{lem-connected-augmented} The adjunction $L:\CDGA_{/\Q}\rightleftarrows \CDGA^1:i$ is a partial Quillen equivalence on the cohomologically connected augmented cdgas weakly of finite type and the connected cdgas weakly of finite type.
\end{lemma}
\begin{proof}
    Since the inclusion $i$ preserves (acyclic) fibrations, this is a Quillen adjunction. In \cite[Rmk. 6.9]{Pepe-Bruno} (setting $G=\{e\}$) it is observed without proof that this is a partial Quillen equivalence: we now fill in the details.

Both full subcategories are closed under weak equivalences, so Condition (1) of \cref{def-partial}  is satisfied. Condition (2) for $i$ is automatic. 

Let us check Condition (3) for $i$. Note the counit of $L\dashv i$ is a natural isomorphism since $i$ is fully faithful. Let $Y\in \CDGA^1$ be weakly of finite type and $q:QY\to Y$ be a cofibrant replacement in $\CDGA_{/\Q}$. We need to check $Lq:LQY\to LY$ is a weak equivalence. Let $m:M\xrightarrow{\sim} Y$ be a minimal model of $Y$, of finite type by hypothesis. We lift this map against the acyclic fibration $q$ to obtain an $u:M\to QY$ such that $qu=m$. Since $L$ preserves weak equivalences between cofibrant objects, $Lu$ is a weak equivalence; since $Y$ and $M$ are connected, we can identify $Lm$ with $m$ via the counit isomorphism. 
By 2-out-of-3 for weak equivalences, since $Lq \circ Lu = m$, we conclude $Lq$ is a weak equivalence.

We now turn to Conditions (2) and (3) for $L$. We will prove that if $X\in \CDGA_{/\Q}$ is cofibrant and cohomologically connected, %
then the unit $\eta_X:X\to LX$ is a weak equivalence. From this, Conditions (2) and (3) for $L$ are immediate. %

Let $m:M\xrightarrow{\sim} X$ be the minimal model of $X$, so $Lm:M\to LX$ is a weak equivalence. By naturality of the adjunction, $Lm$ is the composition $M\xrightarrow{m} X\xrightarrow{\eta_X} LX$. %
Since $m$ and $Lm$ are weak equivalences, then $\eta_X$ is so, too, by 2-out-of-3.
\end{proof}

\begin{remark} Though we shall not use this, we note that in the previous lemma one could omit the hypothesis of weak finite type on both sides and the result would still be valid, with the same proof.
\end{remark}

\begin{theorem} \label{thm-main} There is a zig-zag of adjunctions
\[\begin{tikzcd}
	\DGL & {\CDGC_1} & {(\CDGA^1)^\op} & {\CDGA_{/\Q}^\op} & {\sSet_*}
	\arrow[""{name=0, anchor=center, inner sep=0}, "\C"', shift right=2, from=1-1, to=1-2]
	\arrow[""{name=1, anchor=center, inner sep=0}, "\L"', shift right=2, from=1-2, to=1-1]
	\arrow[""{name=2, anchor=center, inner sep=0}, "{(-)^*}", shift left=2, from=1-2, to=1-3]
	\arrow[""{name=3, anchor=center, inner sep=0}, "{(-)^\circ}", shift left=2, from=1-3, to=1-2]
	\arrow[""{name=4, anchor=center, inner sep=0}, "{i^\op}", shift left=2, from=1-3, to=1-4]
	\arrow[""{name=5, anchor=center, inner sep=0}, "{L^\op}", shift left=2, from=1-4, to=1-3]
	\arrow[""{name=6, anchor=center, inner sep=0}, "{\langle - \rangle}"', shift right=2, from=1-4, to=1-5]
	\arrow[""{name=7, anchor=center, inner sep=0}, "\APL"', shift right=2, from=1-5, to=1-4]
	\arrow["\dashv"{anchor=center, rotate=-90}, draw=none, from=2, to=3]
	\arrow["\dashv"{anchor=center, rotate=-90}, draw=none, from=1, to=0]
	\arrow["\dashv"{anchor=center, rotate=-90}, draw=none, from=4, to=5]
	\arrow["\dashv"{anchor=center, rotate=-90}, draw=none, from=7, to=6]
\end{tikzcd}\]
where $\L\dashv \C$ is a partial crossed derivable equivalence and the other three adjunctions are partial Quillen equivalences. The objects of the full subcategories involved are, from left to right:
\begin{itemize}
    \item The dgls $L$ such that $H_*(L)$ is nilpotent and finite type,
    \item The connected cdgcs weakly of finite type,
    \item The connected cdgas weakly of finite type,
    \item The cohomologically connected augmented cdgas weakly of finite type,
    \item The connected nilpotent rational pointed simplicial sets of finite type.
\end{itemize}
In particular, the $\infty$-categories determined by these objects are all equivalent, where in the cdga cases the opposite $\infty$-categories are understood.

\end{theorem}
\begin{proof}
The statements regarding the first two adjunctions from the left are \cref{prop-LC-partial-mirror} and  \cref{thm-dual-partial}. The next partial Quillen equivalence is the opposite to the one in \cref{lem-connected-augmented}. The last one is \cite[Thm. 10.1]{Bousfield--Gugenheim}, reinterpreted as a partial Quillen equivalence \cite[Thm. 5.4, Rmk. 5.16]{Pepe-Bruno}. Here $\APL$ is Sullivan's functor of piecewise linear differential forms and $\langle -\rangle$ is the realization functor.

The $\infty$-categorical statement follows from \cref{prop-infty}.
\end{proof}

\begin{remark} Let $X$ and $Y$ be connected nilpotent rational pointed simplicial sets of finite type. Let $M_X$, $M_Y$ be their minimal cdga models. Let $L_X=\L(M_X^*)$, $L_Y=\L(M_Y^*)$. Let $\S_*$ be the $\infty$-category of pointed spaces and $\DGL^\infty$ be the $\infty$-category of dgls. We thus obtain a weak equivalence of mapping spaces:
\[\Map_{\S_*}(X,Y)\simeq \Map_{\DGL^\infty}(L_X, L_Y).\]
\end{remark}

\bibliographystyle{alpha}
\bibliography{neisendorfer}
\end{document}